\documentclass[a4paper, oneside, 10pt]{amsart}
\usepackage[headheight=12pt, top=2.5cm, bottom=2cm, inner=2.5cm, outer=2.5cm, foot=1cm]{geometry}
\usepackage[utf8]{inputenc}
\usepackage{amssymb,amsfonts,microtype}
\usepackage[mathscr]{eucal}

\DeclareMathOperator{\ord}{ord}

\DeclareMathOperator{\supp}{supp}
\def\ba#1\ea{\begin{align}#1\end{align}}
\newcommand{\be}{\begin{equation}}
\newcommand{\ber}{\begin{eqnarray}}

\newcommand{\ee}{\end{equation}}
\newcommand{\eer}{\end{eqnarray}}
\newcommand{\Fc}{\mathcal F}
\newcommand{\h}{\mathsf h}

\newcommand{\la}{\langle}

\newcommand{\mdots}{\cdot\ldots\cdot}
\newcommand{\N}{\mathbb N}
\newcommand{\nn}{\nonumber}
\newcommand{\ra}{\rangle}
\newcommand{\Sum}[2]{\sum\limits_{#1}^{#2}}
\newcommand{\Summ}[1]{\sum\limits_{#1}}
\newcommand{\und}{\;\mbox{ and } \;}
\newcommand{\vp}{\mathsf v}
\newcommand{\Z}{\mathbb Z}

\usepackage{aliascnt}
\usepackage[bookmarks=false]{hyperref}
\let\orgautoref\autoref
\renewcommand{\autoref}
        {%
\def\chapterautorefname{Chapter}%
\def\corollaryautorefname{Corollary}%
\def\definitionautorefname{Definition}%
\def\lemmaautorefname{Lemma}%
\def\propositionautorefname{Proposition}%
\def\questionautorefname{Question}%
\def\remarkautorefname{Remark}%
\def\sectionautorefname{Section}%
         \orgautoref}

\newtheorem{theorem}{Theorem}[section]
\newaliascnt{lemma}{theorem}
\newtheorem{lemma}[lemma]{Lemma}
\aliascntresetthe{lemma}
\newaliascnt{corollary}{theorem}
\newtheorem{corollary}[corollary]{Corollary}
\aliascntresetthe{corollary}
\newaliascnt{proposition}{theorem}

\aliascntresetthe{proposition}
\theoremstyle{definition}
\newaliascnt{definition}{theorem}

\aliascntresetthe{definition}
\newaliascnt{remark}{theorem}

\aliascntresetthe{remark}
\newaliascnt{notation}{theorem}

\aliascntresetthe{notation}
\newaliascnt{situation}{theorem}

\aliascntresetthe{situation}
\newaliascnt{example}{theorem}

\aliascntresetthe{example}
\newaliascnt{question}{theorem}

\aliascntresetthe{question}
\newaliascnt{conjecture}{theorem}

\aliascntresetthe{conjecture}



\numberwithin{equation}{section}

\thanks{Supported in part by FWF, the Austrian Science Fund, project number P21576-N18}
\title[]{Arithmetic-Progression-Weighted Subsequence Sums}
\author[D. J. Grynkiewicz, A. Philipp, and V. Ponomarenko]{David J. Grynkiewicz and Andreas Philipp and Vadim Ponomarenko}
\address{University of Graz\\Institute for Mathematics and Scientific Computing\\Heinrichstrasse 36\\8010 Graz\\Austria}
\email{diambri@hotmail.com, andreas.philipp@uni-graz.at, vadim123@gmail.com}

\begin{document}

\begin{abstract}
Let $G$ be an abelian group, let $S$ be a sequence of terms $s_1,s_2,\ldots,s_{n}\in G$  not all contained in a coset of a proper subgroup of $G$, and let $W$ be a sequence of $n$ consecutive integers. Let $$W\odot S=\{w_1s_1+\ldots+w_ns_n:\;w_i\mbox{ a term of } W,\, w_i\neq w_j\mbox{ for } i\neq j\},$$  which is a particular kind of weighted restricted sumset. We show that  $|W\odot S|\geq \min\{|G|-1,\,n\}$, that $W\odot S=G$ if $n\geq |G|+1$, and also characterize all sequences $S$ of length $|G|$ with $W\odot S\neq G$. This result then allows us to characterize when a linear equation $$a_1x_1+\ldots+a_rx_r\equiv \alpha\mod n,$$ where $\alpha,a_1,\ldots, a_r\in \Z$ are given, has a solution $(x_1,\ldots,x_r)\in \Z^r$ modulo $n$ with all $x_i$ distinct modulo $n$.  As a second simple corollary, we also show that there are maximal length minimal zero-sum sequences over a rank $2$ finite abelian group $G\cong C_{n_1}\oplus C_{n_2}$ (where $n_1\mid n_2$ and $n_2\geq 3$) having $k$ distinct terms, for any $k\in [3,\min\{n_1+1,\,\exp(G)\}]$. Indeed, apart from a few simple restrictions, any pattern of multiplicities is realizable for such a maximal length minimal zero-sum sequence.
\end{abstract}

\maketitle

\section{Introduction}

Let $G$ be an abelian group and let $S$  be a sequence of terms from $G$.  It is a classical problem in additive number theory to study which elements from $G$ can be represented as a sum of some subsequence  of $S$ (possibly of predetermined length).
To make this formal, we let $\Sigma(S)$ denote the set of all elements from $G$ that are the sum of terms from some non-empty subsequence of $S$, and we let $\Sigma_n(S)$, where $n\geq 0$ is an integer, denote the set of all elements from $G$ that are the sum of terms from some $n$-term subsequence of $S$.
Throughout this paper, we use the multiplicative standards from \cite{alfredbook} \cite{Alfred-BCN-notes} \cite{GaoGer-survey} for subsequence sum notation, with all formal definitions given in the next section and notation in the introduction kept to a minimum.

The Davenport constant $\mathsf D(G)$, which  is the minimal length of a subsequence from $G$ that guarantees a subsequence with sum zero, i.e., that $0\in \Sigma(S)$, is perhaps the most famous and well-studied subsequence sum question \cite{Rogers-Davenport} \cite{alfredbook}. Other examples include the Erd\H{o}s-Ginzburg-Ziv Theorem \cite{egz} \cite{alfredbook} \cite{natbook}, which states that a sequence $S$ with length $|S|\geq 2|G|-1$ guarantees $0\in \Sigma_{|G|}(S)$, the now proven Kemnitz Conjecture \cite{kemnitz} \cite{alfredbook}, which states that $0\in \Sigma_n(S)$ for $|S|\geq 4n-3$ when $G\cong C_n\oplus C_n$ is a rank $2$ finite abelian group, and the Olson constant, which is analogous to the Davenport Constant only for sets instead of sequences \cite{olsonconstant2} \cite{olsonconstant} \cite{olsonsd}. Related to the Olson Constant is the Critical Number, which is the minimal cardinality of a subset $A$ of $G$ needed to guarantee that \emph{every} element of $G$ can be represented as a sum of distinct elements from $A$ \cite{crit-number}, i.e., that $\Sigma(A)=G$. See \cite{ham-conj} \cite{devos-quad-bound} \cite{vu-complete}  for a handful of more recent results giving bounds for the number of elements representable as a subsequence sum of $S$.

All of the above concerns ordinary subsequence sum questions. Since the establishment of Caro's conjectured weighted Erd\H{o}s-Ginzburg-Ziv Theorem \cite{wegz}, there has been considerable renewed interest to consider various weighted subsequence sum questions \cite{weight-Gao-cyclic} \cite{xia} \cite{thanga-paper} \cite{ordaz-quiroz-cyliccase}
\cite{luca} \cite{hamweightsrelprime} \cite{hamweightegzgroup} \cite{ham-D_A-paper}
\cite{oscar-weighted-projectI} \cite{zhuang-collab} \cite{griffiths} \cite{gao-wegz-partialcase} \cite{adhi0} \cite{adhi1} \cite{adhi2} \cite{adhi3} \cite{adhi4}.
The basic idea is that given a sequence $S$ of terms from an abelian group and a sequence $W$ of integers (or, in the most general form, a sequence of homomorphisms between $G$ and another abelian group $G'$ \cite{homo-weight}), one can instead consider which elements can be represented in the form $w_1s_1+\ldots+w_ns_n$ with the $w_i$ and $s_i$ being the terms of some subsequence from $W$ and $S$, respectively. In this way, the sequence $W$ is viewed as  providing a list of potential weights, and one wishes to know which elements can be represented as a $W$-weighted subsequence sum rather than an ordinary subsequence sum, which is just the case when all terms in the weight sequence $W$ are equal to $1$.
Formally, for a sequence $W=w_1\cdot\ldots\cdot w_n$ of integers $w_i\in \Z$ and an equal length sequence $S=s_1\cdot\ldots\cdot s_n$ with terms $s_i\in G$, we let $$W\odot S=\{w_{\tau(1)}g_1+\ldots+w_{\tau(n)}g_n:\; \tau\mbox{ a permuation of } 1,2,\ldots,n\}.$$ With this notation, the weighted Erd\H{o}s-Ginzburg-Ziv Theorem says that if $W$ is any zero-sum modulo $|G|$ sequence of integers and $S$ is a sequence of terms from $G$ with length $|S|\geq 2|G|-1$, then $S$ has a $|G|$-term subsequence $S'$ with $0\in W\odot S'$. It is still an open conjecture of Bialostocki that the weaker hypothesis $|S|=|G|$ with $S$ zero-sum is enough to guarantee $0\in W\odot S$ when $|G|$ is even \cite{bialostocki} \cite{ap-related}.

If $n=|S|\leq |W|$ and all terms of $W$ are distinct (as will be the case in this paper), so that one may associate $W$ with the set $A:=\supp(W)=\{w_i:\; w_i\mbox{ a term of } W\}$, then $$W\odot S=\{w_1s_1+\ldots +w_ns_n:\,w_i\in A,\,w_i\neq w_j\mbox{ for } i\neq j\}.$$  When all $s_i=1$, then this is precisely the restricted sumset $$A\hat{+}\ldots \hat{+}A=\{a_1+\ldots+a_n:\,a_i\in A,\,a_i\neq a_j\mbox{ for }i\neq j\},$$ which has been extensively studied; see for instance \cite{RS-1} \cite{RS-2} \cite{RS-3} \cite{RS-4} \cite{RS-5} \cite{RS-6}. Thus, for such $W$, studying $W\odot S$  is the same as studying a particular weighted restricted sumset question. In the extreme case when $|A|=n$, there is only one possible element from the restricted sumset $A\hat{+}\ldots \hat{+}A$. However, once the $s_i$ are allowed to take on more general values, the study of such weighted restricted sumsets $W\odot S$ quickly becomes more complicated.

Much of the initial attention regarding weighted subsequence sum problems remained on analogs of the Davenport Constant and Erd\H{o}s-Ginzburg-Ziv Theorem, often providing results valid when both sequences $W$ and $S$ are arbitrary, the idea being that restricting such results to the case when $W$ is the constant $1$ sequence gives an extension of more classical subsequence sum questions. The weighted Erd\H{o}s-Ginzburg-Ziv Theorem mentioned above gives one such example. However, there is a very  natural non-constant weight sequence that has not yet been much studied: namely, one can consider $W$-weighted subsequence sums of $S$ when $W$ is an arithmetic progression of integers.
The focus of this paper is to investigate such weighted subsequence sums. In particular, since the terms of $W$ are generally all distinct, this is also a particular type of weighted restricted sumset question as discussed above.

Indeed, the main goal is to show that $|G|+1$ is the minimal length of a sequence $S$ from a finite abelian group $G$ needed to guarantee that every element of $G$ is representable as a $W$-weighted subsequence sum, where $W$ is an arithmetic progression of $|S|$ consecutive integers (provided the terms of $S$ do not all come from a coset of a proper subgroup, which is easily seen to be a necessary condition for $W\odot S=G$ to hold). Moreover, we also characterize the structure of those sequences of length one less which do not realize every element of $G$  as a $W$-weighted subsequence sum and  give a lower bound for $|W\odot S|$ in terms of $|S|$, which, at least in rather limited special cases, is tight (simply consider $S=0^{|S|-1}g$ with $g$ a generator of $G$). In the notation of the following section, our main result is as follows. It is worth noting that \autoref{main} contains, as a very special case, the main result from \cite{ap-related}, which was devoted to proving the aforementioned conjecture of Bialostocki in the case when the weight sequence is an arithmetic progression of even difference.

\begin{theorem}
\label{main}
Let $G$ be a finite abelian group,
let $S$ be a sequence of terms from $G$ not all contained in a coset of a proper subgroup, and let $W$ be a sequence of $|S|$ consecutive integers.
\begin{itemize}

\item $|W\odot S|\geq \min\{|G|-1,\,|S|\}$.
\item If $|S|\geq |G|+1$, then $W\odot S=G$. Indeed, $W'\odot S'=G$ for some subsequence $S'\mid S$ with $|S'|=|G|$, where $W'=(0)(1)\cdot\ldots\cdot (|G|-1)\in \Fc(\Z)$.
\item If $|S|=|G|$ and $W\odot S\neq G$, then $|G|\geq 3$ and  either
\begin{itemize}
\item[(i)] $G\cong C_2\oplus C_2$, $|\supp(S)|=|S|=|G|=4$ and $W\odot S=G\setminus \{0\}$, or
\item[(ii)]$G$ is cyclic, $(-g'+S)=0^{|G|-2}(g)(-g)$, for some $g,\,g'\in G$ with $\ord(g)=|G|$, and $W\odot S=G\setminus\lbrace\frac12(|G|-1)|G|g'\rbrace$. In particular, $W\odot S$ contains every generator $h\in G$.
\end{itemize}
\end{itemize}
\end{theorem}

In the final sections, we give simple corollaries of the above theorem first regarding whether a linear equation has a solution modulo $n$ with all members of the solution distinct modulo $n$, and then concerning the pattern of multiplicities possible in a maximal length minimal zero-sum sequence over a rank $2$ finite abelian group, thus providing more refined information than immediately available from the recent characterization of such sequences \cite{MR1985670} \cite{propB-2} \cite{propB-3} \cite{propB-4} \cite{propB-5}.

\section{Preliminaries}

Our notation and terminology are consistent with \cite{alfredbook} \cite{Alfred-BCN-notes} \cite{GaoGer-survey}. We briefly gather some key notions and fix the notation concerning sequences and sumsets over finite abelian groups. Let $\N$ denote the set of positive integers and let $\N_0=\N\cup\lbrace 0\rbrace$. For $a,\,b\in\Z$, we set $[a,b]=\lbrace x\in\Z: a\leq x\leq b\rbrace$. Throughout, all abelian groups will be written additively. We let  $C_n$ denote a cyclic group with $n$ elements.

Let $G$ be a finite abelian group, $H\leq G$ a subgroup and $A\subseteq G$ a subset. We use  $\phi_H:\rightarrow G/H$ to denote the canonical homomorphism and  let $\la A\ra_*=\la A-A\ra$ denote the minimal subgroup $\la A\ra_*$ for which $A$ is contained in a $\la A\ra_*$-coset. Note that $\la A\ra_*=\la A-a\ra$ for any $a\in A$.

For subsets $A,\,B\subseteq G$, we set $$A+B=\lbrace a+b:a\in A,\,b\in B\rbrace$$ for their \emph{sumset} and, if $B=\lbrace b\rbrace$,  write $A+B=A+b=\lbrace a+b:a\in A\rbrace$. We write $$\mathsf H(A)=\lbrace g\in G:g+A=A\rbrace$$ for the \emph{stabilizer} of $A$, which is in fact a subgroup of $G$ for finite $A$.
If $A$ is a union of $H$-cosets, for some subgroup $H\leq G$, then we say $A$ is \emph{$H$-periodic}, which is equivalent to saying $H\leq \mathsf H(A)$, i.e, that $A+H=A$.
We call $A$ \emph{periodic} if $\mathsf H(A)$ contains a nontrivial subgroup, and otherwise $A$ is \emph{aperiodic}.

We use $\mathcal F(G)$ to denote all finite length (unordered) sequences with terms from $G$, refer to the elements of $\mathcal F(G)$ simply as sequences, and write all such sequences multiplicatively, so that a sequence $S\in\mathcal F(G)$ is written in the form
\[
S=g_1\cdot\ldots\cdot g_l=\prod_{g\in G}g^{\vp_g(S)},\quad\mbox{with }\vp_g(S)\in\N_0\quad\mbox{for all }g\in G.
\]
We call $\vp_g(S)$ the \emph{multiplicity} of $g$ in $S$ and say that $S$ contains $g$ if $\vp_g(S)>0$. The notation $S_1\mid S$ indicates that $S_1$ is a subsequence of $S$, that is, $\vp_g(S_1)\leq\vp_g(S)$ for all $g\in G$. If a sequence $S\in\mathcal F(G)$ is written in the form $S=g_1\mdots g_l$, we tacitly assume that $l\in\N_0$ and $g_1,\ldots,g_l\in G$. A sequence of finite, nonempty subsets of $G$ is called a \emph{setpartition}.

For a sequence
\[
 S=g_1\mdots g_l=\prod_{g\in G}g^{\vp_g(S)}\in\mathcal F(G)
\]
and $n\in\N$, we call
\begin{align*}
 |S|=l=\sum_{g\in G}\vp_g(S)\in\N_0\quad &\quad\mbox{the \emph{length} of $S$,}\\
 \sigma(S)=\sum_{i=1}^lg_i=\sum_{g\in G}\vp_g(S)g\in G\quad &\quad\mbox{the \emph{sum} of $S$,}\\
 \Sigma_n(S)=\left\lbrace\sum_{i\in I}g_i:\; I\subseteq[1,l],\,|I|=n\right\rbrace\subseteq G\quad &\quad\mbox{the \emph{set of $n$-term subsequence sums} of $S$,}\\
 \supp(S)=\lbrace g_1,\ldots,g_l\rbrace=\lbrace g\in G:\vp_g(S)>0\rbrace\quad &\quad\mbox{the \emph{support} of $S$, and}\\
 \h(S)=\max\lbrace\vp_g(S):g\in G\rbrace\quad &\quad\mbox{the \emph{maximum multiplicity} of a term of $S$.}
\end{align*}
For $g'\in G$, we write
\[
 (g'+S)=(g'+g_1)\mdots (g'+g_l)=\prod_{g\in G}(g'+g)^{\vp_g(S)}=\prod_{g\in G}g^{\vp_{g-g'}(S)}\in\mathcal F(G).
\]
The sequence $S$ is called
\begin{itemize}
\item a \emph{zero-sum sequence} if $\sigma(S)=0$,
\item \emph{zero-sum free} if there is no non-trivial zero-sum subsequence, and
\item a \emph{minimal zero-sum sequence} if $|S|>0$, $\sigma(S)=0$, and every subsequence $S'\mid S$ with $0< |S'|<|S|$ is zero-sum free.
\end{itemize}
The Davenport constant $\mathsf D(G)$ of $G$ is then the smallest integer $l\in\N$ such that every sequence $S$ over $G$ of length $|S|\geq l$ has a non-trivial zero-sum subsequence (equivalently, $S$ is not zero-sum free).

The following is one of the foundational results of set addition. Note that multiplying both sides of the inequality from Kneser's Theorem \cite{kt} \cite{natbook} \cite{alfredbook} by $|H|$ yields $$|\Sum{i=1}{n}A_i|\geq \Sum{i=1}{n}|A_i+H|-(n-1)|H|=\Sum{i=1}{n}|A_i|-(n-1)|H|+\rho,$$ where
$\rho:=\Sum{i=1}{n}|(A_i+H)\setminus A_i|$ is the number of $H$-holes in the sets $A_i$. Additionally, if $\Sum{i=1}{n}A_i$ is aperiodic, then Kneser's Theorem implies $$|\Sum{i=1}{n}A_i|\geq \Sum{i=1}{n}|A_i|-n+1.$$

\begin{theorem}[Kneser's Theorem]
Let $G$ be an abelian group, let $A_1,\ldots,A_n\subseteq G$ be finite, nonempty subsets, and let $H=\mathsf H(\Sum{i=1}{n}A_i)$. Then
$$|\Sum{i=1}{n}\phi_H(A_i)|\geq \Sum{i=1}{n}|\phi_H(A)|-n+1.$$
\end{theorem}

We will also need the following simple consequence of the Pigeonhole Principle \cite{natbook}.

\begin{lemma}\label{pigeonhole-lemma}
Let $G$ be a finite abelian group and let $A,\,B\subseteq G$ be nonempty subsets. If
$|A|+|B|-1\geq |G|$, then $A+B=G$.
\end{lemma}

\section{Proof of \autoref{main}}
\label{mainsec}

For two sequences $W\in\Fc(\Z)$ and $S\in \Fc(G)$, where $G$ is an abelian group, set $$W\odot S=\{w_1g_1+\ldots+w_rg_r:\,w_1\cdot\ldots\cdot w_r\mid W,\; g_1\cdot\ldots\cdot g_r\mid S\und r=\min\{|W|,\,|S|\}\}.$$
Note that $$W\odot S=(W0^{|S|-r})\odot (S0^{|W|-r})\quad\mbox{ with }\quad |W0^{|S|-r}|=|S0^{|W|-r}|=\max\{|W|,\,|S|\},$$ where $r=\min\{
|W|,\,|S|\}$.
 Also, if $|W|\geq |S|$, then \be(W+w)\odot S=W\odot S+w\sigma(S)\quad \mbox{ for all }\; w\in \Z,\label{invariant-W-shift}\ee while if $|S|\geq |W|$, then \be\label{invar-S-shift}
W\odot (S+g)=W\odot S+\sigma(W)g\quad \mbox{ for all }\; g\in G.\ee
In particular, if $|W|=|S|$, then $G=W\odot S$ if and only if $G=(W+w)\odot(S+g)$ for all $w\in \Z$ and $g\in G$.

We begin with a lemma dealing with the case $|S|=3$ for \autoref{main}.

\begin{lemma}\label{key-lemma}
Let $G$ be an abelian  group, let $W=(0)(1)\cdot\ldots\cdot(|W|-1) \in\Fc(\Z)$ be a sequence of consecutive integers, let $x,\,y\in G\setminus \{0\}$ be  nonzero elements with $\langle x,\,y\rangle=G$, and set $S=xy\in \Fc(G)$.
\begin{itemize}
\item[(i)] If $|W|\geq 3$, then
 $\langle W\odot S\rangle_*=G$.
 \item[(ii)] If $x=y$, then $|W\odot S|\geq \min\{|G|,\,2|W|-3\}.$
\item[(iii)] If $x\neq y$, then $|W\odot S|\geq \min\{|G|-1,\,2|W|-2\}.$
\end{itemize}
\end{lemma}

\begin{proof}
If $|W|\leq 2$, then the lemma is easily verified. So we may assume $|W|\geq 3$. In this case, $x,\,2x,\,2x+y\in W\odot S$, so that $$\langle W\odot S\rangle_*\supseteq\langle x,2x,2x+y\rangle_*=\langle x,x+y\rangle =\langle x,y\rangle =G,$$ whence $\langle W\odot S\rangle_*=G$ follows, yielding (i).
If  $x=y$, then $$W\odot S=\{x+0,2x+0,\ldots,(|W|-1)x+0,(|W|-1)x+x,\ldots, (|W|-1)x+(|W|-2)x\},$$ from which (ii) is readily deduced. Therefore
it remains to prove the lower bound for $|W\odot S|$ when $x\neq y$.

Without loss of generality, assume $\ord(x)\geq \ord(y)$. Let $r=|W|\geq 3$ and set $H=\langle x\rangle$. Since $G/H=\langle \phi_H(y)\rangle$, it follows that $$|H|=\ord(x)\geq \ord(y)\geq \ord(\phi_H(y))=|G/H|.$$
Now we have \be\label{rubbter}
W\odot S=\left\{
           \begin{array}{lllllll}
             \Box & 0+y & 0+2y & &\cdots & & 0+(r-1)y \\
             x & \Box & x+2y & &\cdots & &x+(r-1)y \\
             2x & 2x+y& \Box & &\cdots & &2x+(r-1)y \\
             3x & 3x+y & 3x+2y & &\cdots & &3x+(r-1)y \\
             \vdots & \vdots & \vdots & & & & \vdots \\
             (r-1)x & (r-1)x+y & (r-1)x+2y & &\cdots  & & \Box \\
           \end{array}
         \right\}.\ee Note that each column consists of elements from the same $H$-coset. We divide the remainder of the proof into several cases based off the number of $H$-cosets in $G$.

\subsection*{Case 1:} $|G/H|\geq 3$. If $r\leq |G/H|\leq |H|=\ord(x)$, then all columns in \eqref{rubbter} correspond to distinct $H$-cosets filled with distinct elements, whence $|W\odot S|=r(r-1)\geq 2r-2$. If $|G/H|+1\leq r\leq |H|=\ord(x)$, then the
first $|G/H|$ columns in \eqref{rubbter} are distinct and each contain at least $r-1$ elements, whence $|W\odot S|\geq (r-1)|G/H|\geq 3r-3\geq 2r-2$. Finally, it remains to consider the case $r>|H|=\ord(x)$, for which $\ord(x)=|H|$ must be finite. Let $r=|H|+s$ with $s\geq 1$. In this case, we see that the first $|G/H|$ columns cover all distinct $H$-cosets and are each missing at most one element, while the first $s$ of these columns are missing no elements. In consequence, $|W\odot S|\geq (|H|-1)|G/H|+\min\{|G/H|,\,s\}$. If $s\geq |G/H|$, then $|W\odot S|\geq |G|$ follows, as desired. Otherwise, when $1\leq s\leq |G/H|-1\leq |H|-1$, we can recall that $r=|H|+s$ and $|G/H|\geq 3$ and thus conclude that \ba\nn|W\odot S|&\geq (|H|-1)|G/H|+s\geq 3|H|-3+s=r-2+|H|+(|H|-1)\\\nn&\geq
r-2+|H|+s= 2r-2,\ea also as desired.

\subsection*{Case 2:} $|G/H|=2$. In this case, since $r\geq 3>|G/H|$, we see that the first two columns of \eqref{rubbter} cover both distinct $H$-cosets. If $r\leq \ord(x)=|H|$, then there are $r-1$ elements in both these columns, whence $|W\odot S|\geq 2(r-1)$, as desired. On the other hand, if $r\geq \ord(x)+1$, then the first column is missing no element while the second column is missing at most one, whence $|W\odot S|\geq |G|-1$, also as desired.

\subsection*{Case 3:} $|G/H|=1$. In this case, $x$ generates $G$, and thus $y=\alpha x$ for some $\alpha\in \Z$ with $\alpha\in (-\frac{n}{2}, \lfloor\frac{n+1}{2}\rfloor]$, where $n:=\ord(x)=|G|$.
 It suffices to prove (iii) when $$|W|=r\leq \left\lceil\frac{n+1}{2}\right\rceil,$$ as for larger $|W|$, one can simply apply (iii) using $r=\left\lceil\frac{n+1}{2}\right\rceil$ and note that $2r-2\geq n-1=|G|-1$ holds in this case. Thus, in view of $r\geq 3$, it follows that $n=|G|\geq 4$. To simplify notation, we may assume $x=1$ generates the cyclic group $G\cong C_n$.

Now, from \eqref{rubbter}, we know that $\{1,2,\ldots,(r-1)\}\subseteq W\odot S$.
We also have \be\label{cuyo}\{0+\alpha,\;2+\alpha,\; 3+\alpha,\;\ldots,(r-1)+\alpha \}\subseteq W\odot S.\ee Note that $r-1+\alpha\leq \lceil\frac{n+1}{2}\rceil-1+\lfloor\frac{n+1}{2}\rfloor=n$. Thus, if $\alpha\geq r$, then the elements from \eqref{cuyo} will be disjoint from $\{1,2,\ldots,(r-1)\}\subseteq W\odot S$, whence $|W\odot S|\geq 2(r-1)$, as desired. Likewise, if $\alpha\leq -(r-1)$, then we have $n+\alpha\geq \frac{n+1}{2}>r-1$, and the elements in \eqref{cuyo} will again be disjoint from $\{1,2,\ldots,(r-1)\}\subseteq W\odot S$, yielding the desired bound $|W\odot S|\geq 2(r-1)$ once more. Thus, in both cases, (iii) holds, and we may now  assume \be\label{stungle}-r+2\leq \alpha\leq r-1.\ee

\medskip

Suppose $\alpha\geq 0$. Then, in view of $y\neq x$, $y\neq 0$ and \eqref{stungle}, we have $\alpha\in [2,r-1]$. The sums $0+1,0+2,\ldots,0+r-1\in W\odot S$ show that $[1,r-1]\subseteq W\odot S$. The sums $j\cdot \alpha+(r-\alpha+i)$, for $j\in [1,r-1]$ and $i\in [0,\alpha-1]\setminus \{j-r+\alpha\}$, show that each interval $[r+(j-1)\alpha,r+j\alpha-1]$ is contained in $W\odot S$ apart from possibly the element $j\alpha+(r-\alpha+i)=j\alpha+j$ when $i=j-r+\alpha\in [0,\alpha-1]$, for $j\in [1,r-1]$.
In particular, in order for an element to be missing from the interval $[r+(j-1)\alpha,r+j\alpha-1]$ in $W\odot S$, we must have $j-r+\alpha\geq 0$, i.e., $j\geq r-\alpha$. As a result, we conclude from all of the above that
$$[1, (r-\alpha+1)\alpha+r-\alpha]\setminus \{(r-\alpha)\alpha+(r-\alpha)\}\subseteq W\odot S,$$
from which, in view of $\alpha\in [2,r-1]$ and $r\geq 3$, it is easily deduced that
\ba |W\odot S|&\geq \min \{|G|-1,\;(r-\alpha+1)\alpha+r-\alpha-1\}\\ &\geq
 \min\{|G|-1,\; 3r-5,\;2r-2 \}\geq \min\{|G|-1,\;\;2r-2\},\ea  as desired. So we now assume $\alpha<0$.

\medskip

Since $\alpha<0$, we infer from \eqref{stungle} that $\alpha\in [-r+2,-1]$. Furthermore, \eqref{stungle} also gives \be\label{rradic}r\geq |\alpha|+2.\ee
If $\alpha=-1$, then we clearly have
$$[-(r-1),-1]\cup [1,r-1]=([1,r-1]\odot (-1)+0\cdot 1)\cup (0\cdot(-1)+[1,r-1]\odot 1)\subseteq W\odot S,$$ from which (iii) easily follows.
If $\alpha=-2$, then $[1,r-1]=0\cdot (-2)+[1,r-1]\odot 1\subseteq W\odot S$ and \begin{align*}&\{1\cdot (-2)+2=0,\quad1\cdot (-2)+0=-2,\quad \\&2\cdot (-2)+1=-3,\quad 2\cdot (-2)+0=-4,\quad 3\cdot (-2)+1=-5,\quad3\cdot(-2)+0=-6,\quad\ldots,\\&(r-1)\cdot (-2)+1=-2r+3,\quad (r-1)\cdot (-2)+0=-2r+2\}\subseteq W\odot S.\end{align*} Consequently, $$[-2r+2,r-1]\setminus\{-1\}\subseteq W\odot S,$$ from which it is easily deduced that $|W\odot S|\geq \min\{|G|-1,\;3r-3\}\geq \min\{|G|-1,\;2r-2\}$, as desired.
Therefore we may assume $\alpha\leq -3$, in which case \eqref{rradic} gives $$r\geq |\alpha|+2\geq 5.$$

We know $[1,r-1]=0\cdot \alpha +[1,r-1]\odot 1\subseteq W\odot S$.  Since $\alpha\leq -3$ and $3\leq |\alpha|\leq r-2$, we also have $1\cdot \alpha +|\alpha|\cdot 1=0\in W\odot S$, whence $$[0,r-1]\subseteq W\odot S.$$
Next we claim that, for each $j\in [1,r-1]$, $W\odot S$ also contains all elements  from $[j\alpha,(j-1)\alpha-1]$ except possibly $j\alpha+j$. Indeed, to see this, we have only to note that  $j\cdot \alpha+\beta \cdot 1\in W\odot S$ for $\beta\in [0,|\alpha|-1]\setminus \{j\}$. Next, since $\alpha\leq -2$, it follows  that $$j\alpha+j=(j+1)\cdot \alpha +(|\alpha|+j)\cdot 1\in W\odot S \quad \mbox{ for }\;j\leq r-1-|\alpha|.$$ As a result, we conclude from the above work that $$[(r-|\alpha|+1)\alpha+(r-|\alpha|+1)+1,r-1]\setminus \{(r-|\alpha|)\alpha+(r-|\alpha|)\}\subseteq W\odot S,$$ which, combined with $|\alpha|\in [3,r-2]$ and $r\geq 5$, allows us to easily infer that \ba\nn|W\odot S|&\geq \min\{|G|-1,\; (r-|\alpha|+2)|\alpha|-3\}\\\nn &\geq
\min\{|G|-1,\; 3r-6,\;4r-11\}\geq \min\{|G|-1,\,2r-2\},\ea completing the proof.
\end{proof}

We will need the following technical refinement of the case $|W|=3$ from  \autoref{key-lemma}.

\begin{lemma}\label{special-lemma} Let $G$ be an abelian  group with $|G|\geq 5$, let $W=(0)(1)(2) \in\Fc(\Z)$ be a sequence of $3$ consecutive integers, let $x,\,y,\,z\in G$ be  distinct elements with $\langle x,\,y,\,z\rangle_*=G$, and set $S=xyz\in \Fc(G)$.
Suppose $\ord(x-z),\,\ord(y-z),\,\ord(x-y)\geq 3$.
Then there exists a subset $X\subseteq W\odot S$ with $|X|=4$, $|X\cap (3z+\langle x,\,z\rangle_*)|\geq 2$ and $\langle X\rangle_*=G$. Furthermore, if $G\not\cong C_6$, then $|\mathsf H(X)|\neq 2$.
\end{lemma}

\begin{proof} In view of \eqref{invar-S-shift}, we can w.l.o.g.~translate $S$ so that  $z=0$.
If the three terms of $S$ are in arithmetic progression, say  $S=0(x)(2x)$, $S=0(y)(2y)$ or $S=(-x)0(x)$,
then $W\odot S=\{1,2,4,5\}\odot x$, $W\odot S=\{1,2,4,5\}\odot y$ or $W\odot S=\{-2,-1,1,2\}\odot x$, and the lemma is easily verified taking $X=W\odot S$.
Therefore we may assume $S$ is not in arithmetic progression, whence
\be\label{stuff} y\notin \{-x,0,x,2x\}\quad \und \quad x\notin \{-y,0,y,2y\}.\ee
Consider the set $X:=\{x,2x,2x+y,y\}\subseteq W\odot S$. In view of \eqref{stuff} and $\ord(x)\geq 3$, we have $|X|=4$. We also have $\langle x,2x,2x+y\rangle_*=\langle x,x+y\rangle=\langle x,y\rangle=\langle x,y,z=0\rangle_*=G$, so that $\langle X\rangle_*=G$. Clearly, $|X\cap \langle x\rangle|\geq 2$.

Finally, if $|\mathsf H(X)|=2$, then there must be a pairing up of the $4$ elements of $X$ such that the difference of elements in each pairing is equal to the same order two element. There are three such possible pairings: $\{x,2x\}$ and $\{y,2x+y\}$;  $\{x,y\}$ and $\{2x,2x+y\}$; $\{x,2x+y\}$ and $\{y,2x\}$. Since $\ord(x)\geq 3$ and $\ord(y)\geq 3$, we cannot have $x$ and $2x$, nor $2x$ and $2x+y$, being in the same cardinality two coset, which rules out the first two possible pairings. On the other hand, if  $\{x,2x+y\}$ and $\{y,2x\}$ are both cosets of the same order $2$ subgroup, then we must have $x+y=(2x+y)-x=2x-y$, contradicting \eqref{stuff}. As this exhausts all possible pairings, we conclude that $|\mathsf H(X)|=2$ does not hold, completing the proof.
\end{proof}



Next, we show that if the terms of $S$ generate $G$ (up to translation), then so do the elements of $W\odot S$.

\begin{lemma}\label{lemma-generation}  Let $G$ be an abelian group, let $S\in \Fc(G)$ be a sequence, and let $W\in \Fc(\Z)$ be a sequence of consecutive integers. If $|W|=|S|$, then $\la W\odot S\ra_*=\la \supp(S)\ra_*$.
\end{lemma}

\begin{proof}
In view of \eqref{invar-S-shift}, \eqref{invariant-W-shift} and $|W|=|S|$, there is no loss in generality if we translate $W$ and $S$ such that $W=(0)(1)\cdot\ldots \cdot (|S|-1)$ and $0\in \supp(S)$. If $|S|\leq 2$, then the lemma is easily verified. We proceed by induction on $|S|$. If $\supp(S)=\{0\}$, then $\la \supp(S)\ra_*=\{0\}=\la W\odot S\ra_*$. Therefore we may assume $|\supp(S)|\geq 2$. We trivially have $\la W\odot S\ra_*\subseteq \la \supp(S)\ra=\la \supp(S)\ra_*$, with the latter equality in view of $0\in \supp(S)$. Therefore, it suffices to show the reverse inclusion $\la \supp(S)\ra_*\subseteq \la W\odot S\ra_*$.

Let $x\in \supp(S)$ be nonzero. Let $K:=\langle  \supp(Sx^{-1})\rangle$. Since $0\in \supp(Sx^{-1})$, we have $$K=\la \supp(Sx^{-1})\ra_*=\la \supp(Sx^{-1})\ra.$$ Thus, by induction hypothesis, we conclude that  $$\la \left(0\cdot x\right)+\left( W0^{-1}\odot Sx^{-1}\right)\ra_*=K;$$ moreover, since  $R\odot (Sx^{-1})\subseteq \la \supp(Sx^{-1})\ra=K$ for any sequence of integers $R\in\Fc(\Z)$, we actually have $$\left(0\cdot x\right)+\left( W0^{-1}\odot Sx^{-1}\right)\subseteq K.$$ Consequently, to show $\la \supp(S)\ra_*\subseteq \la W\odot S\ra_*$, it suffices to show that $W\odot S$ contains some element from $x+K$. However, clearly $$(1\cdot x)+\left((0)(2)(3)\cdot\ldots\cdot (|S|-1)\odot Sx^{-1}\right)\subseteq W\odot S$$ is a nontrivial subset of $x+\la \supp(Sx^{-1})\ra=x+K$, so that $W\odot S$ indeed contains some element from $x+K$, completing the proof.
\end{proof}

The following lemma can be found in \cite{quasi-kemp} as observation (c.5). See \cite[Proposition 5.2]{phd-thesis} for a more detailed proof.

\begin{lemma}\label{punctured-lem}
Let $G$ be an abelian group, let $A\subseteq G$ be a finite, nonempty subset, and let $x\in G\setminus A$. If $A\cup\{x\}$ is $H$-periodic with $|H|\geq 3$, then $A\cup \{y\}$ is aperiodic for every $y\in G\setminus \{x\}$.
\end{lemma}

We now proceed with the proof of our main result.

\begin{proof}[Proof of \autoref{main}]
In view of \eqref{invar-S-shift} and \eqref{invariant-W-shift}, our problem is invariant when translating $S$ or $W$, so we may w.l.o.g.~assume $0\in \supp(S)$ is a term with maximum multiplicity $\vp_0(S)=\mathsf h(S)$.
For $|G|\leq 4$, the theorem is quickly verified by an exhaustive enumeration of all possible sequences.
Likewise when $|S|\leq 2$, while the case $|S|=3$ follows from \autoref{key-lemma}(ii)--(iii). Therefore we may assume $$|G|\geq 5\quad\und\quad |S|\geq 4$$ and
proceed by a double induction on $(|G|,|S|)$, assuming the theorem proved for any sequence over a smaller cardinality subgroup as well as any sequence over $G$ with smaller length than $S$.

In view of \eqref{invar-S-shift},  we see that if $(-g'+S)=0^{|G|-2}(g)(-g)$, for some $g,\,g'\in G$ with $\ord(g)=|G|$, then
$W\odot S=G\setminus\lbrace\frac12(|G|-1)|G|g'\rbrace$; in particular, $W\odot S$ contains every generator $h\in G$ in view of $|G|\geq 3$. Thus the latter conclusions of (ii) are simple consequences of the structural characterization of $S$ given there.

Next let us show that the structural characterization from the third part of the theorem implies the second part of the theorem. Indeed, if $|S|=|G|+1$ and $W'\odot S0^{-1}\neq G$, then recalling that $|G|\geq 5$ and applying the characterization to $S0^{-1}$ yields  $S={g'}^{|G|-2}(g'+g)(g'-g)0$ for some $g,\,g'\in G$ with $\ord(g)=|G|$.
Since $\ord(g)=|G|\geq 5$, we have $(g'+g)\neq (g'-g)$. Thus, if $g'\neq 0$, then $|G|-2\leq \mathsf h(S)=\vp_0(S)\leq 2$, contradicting that $|G|\geq 5$. Therefore we conclude that $S={0}^{|G|-1}(g)(-g)$ with $\ord(g)=|G|$, and now clearly the subsequence $S'=0^{|G|-1}g$ has $W'\odot S'=G$.
So we see that it suffices to prove the first and third parts of the theorem. In particular, we can assume $|S|\leq |G|$ and we need to show either $|W\odot S|\geq |S|$ or else $|S|=|G|$ with $S$ being described by (ii).

\subsection*{Case 1:} $|\supp(S)|=2$.

In this case, in view of $\langle\supp(S)\rangle=G$, we have $S=0^{|S|-\alpha}g^\alpha$ with $\ord(g)=|G|$ and $1\leq \alpha\leq |S|-1\leq |G|-1$.
As a result, it is easily seen that $W\odot S$ is an arithmetic progression with difference $g$ and length $$|W\odot S|=\min\{|G|,\,|\Sigma_\alpha([0,|S|-1])|\}=\min\{|G|,\,\alpha |S|-\alpha^2+1\}\geq |S|,$$ where the final equality follows in view of $1\leq \alpha\leq |S|-1\leq |G|-1$. Thus $|W\odot S|\geq |S|$, as desired. This completes Case 1.

\subsection*{Case 2:} $\h(S)\geq |S|-2$.

Since $\langle\supp(S)\rangle=G$ with $|G|\geq 5$, we trivially have $\h(S)\leq |S|-1$. If $\h(S)= |S|-1$, then $\langle\supp(S)\rangle=G$ and $\vp_0(S)=\mathsf h(S)$ ensure that $S=0^{|S|-1}g$ with $\ord(g)=|G|$, and now Case 1 completes the proof. So it remains to consider $\h(S)=|S|-2$ for Case 2. In this case, $S=0^{|S|-2}xy$ with $x,\,y\in G\setminus\{0\}$. In view
 of Case 1, we may assume $x\neq y$. Note \be\label{dealth}(W(|S|-1)^{-1}\odot T)\cup (W0^{-1}\odot T)\subseteq W\odot S,\ee where $T:=xy\in \Fc(G)$.
\autoref{key-lemma}(iii) and $|S|\geq 4$ together imply that $$|(W(|S|-1)^{-1})\odot T|\geq \min\{|G|-1,\,2|W|-4\}=\min\{|G|-1,\,2|S|-4\}=\min\{|G|-1,\,|S|\}.$$ In consequence, if $|S|\leq |G|-1$, then the proof is complete, so we assume $|S|=|G|$. In this case, we have $$|W(|S|-1)^{-1}\odot T|\geq |G|-1$$ and likewise $|W0^{-1}\odot T|\geq |G|-1$. Combined with \eqref{dealth}, we once more obtain the desired conclusion $W\odot S=G$ unless $W0^{-1}\odot T=W(|S|-1)^{-1}\odot T$ with $|W0^{-1}\odot T|=|G|-1$. In particular, $W0^{-1}\odot T$ is aperiodic, in which case \eqref{invariant-W-shift} shows that $W0^{-1}\odot T=W(|S|-1)^{-1}\odot T$ is only possible if $\sigma(T)=x+y=0$. Thus $y=-x$.
We now know $S=0^{|G|-2}x(-x)$. Hence, since $\langle \supp(S)\rangle=G$, we conclude that $x$ generates $G$, whence $G$ is cyclic with  $\ord(x)=|G|$, which gives the desired conclusion of (ii). This completes Case 2.

\subsection*{Case 3:} There exists a
subsequence $T\mid S$ with $\langle \supp(T)\rangle_*=H$, where $H<G$ is a proper, nontrivial subgroup, and either  $|T|\geq |H|+1$
(if $|H|\geq 3$) or $|T|\geq |H|$ (if $|H|=2$).


Let $W_T=(0)(1)\cdot\ldots\cdot (|H|-1)\in \Fc(\Z)$. By induction hypothesis, we can apply the theorem to $T$ to conclude that $W_{T}\odot T'$ is an $H$-coset for some subsequence $T'\mid T$ with $|T'|=|H|$. By translating appropriately, we can w.l.o.g.~assume $0\in \supp(T')$, though we may lose that $\h(S)=\vp_0(S)$. Let $$\la \supp(\phi_H(S{T'}^{-1}))\ra_*=K/H,\quad \mbox{ where }\; H\leq K\leq G.$$ Then all terms of $\phi_H(S{T'}^{-1})$ are contained in a single $K/H$-coset, say $\supp\{\phi_K(S{T'}^{-1})\}=\{\phi_K(\alpha)\}$, where $\alpha\in G$.
Consequently, since $\langle \supp(S)\rangle_*=\langle\supp(S)\rangle=G$, so that $\langle \supp(\phi_K(S))\rangle=G/K$, and since $\supp(T')\subseteq H\subseteq K$, so that
$\supp(\phi_K(T'))=\{0\}$, it follows that  \be\label{spaff}\langle \phi_K(\alpha)\rangle=G/K.\ee

If $T\neq T'$, which holds whenever $|H|\geq 3$, then it follows in view of $\supp(\phi_H(T))=\{0\}$ that  $\supp(\phi_H(S{T'}^{-1}))=\supp(\phi_H(S))$, whence $\langle \supp(\phi_H(S{T'}^{-1}))\rangle_*=\langle \supp(\phi_H(S))\rangle_*=G/H$. In summary, \be\label{KG} K=G\quad\mbox{ when }\;T'\neq T\; \mbox{ or }\; |H|\geq 3.\ee


Next, let us show that \be\label{H-biggish} |W\odot S|\geq 2|H|.\ee
If $|\supp(\phi_H(S{T'}^{-1}))|\geq 2$, then $|WW_T^{-1}\odot \phi_H(S{T'}^{-1})|\geq 2$, which combined with the fact that $W_T\odot T'$ is an $H$-coset yields \eqref{H-biggish}. Therefore assume instead $\supp(\phi_H(S{T'}^{-1}))=\{\phi_H(\beta)\}$, where $\beta\in \supp(S{T'}^{-1})$. Since $\supp(T)\subseteq H$, if $\phi_H(\beta)=0$, then $\supp(S)\subseteq H<G$ follows, contradicting that $\la \supp(S)\ra=G$. Therefore $\phi_H(\beta)\neq 0$. However, if $|H|\geq 3$, then $S{T'}^{-1}$ contains a term from $T$, and thus a term from $H$, in which case $\phi_H(\beta)=0$, contrary to what we just noted. Therefore we can now assume $|H|=|T|=2$ for proving \eqref{H-biggish}.
Now
$(x+W_T)\odot T'=H$ for all $x\in [0,|S|-2]$. Thus, if \eqref{H-biggish} fails, then we must have \be\label{cupsache}|\bigcup_{x\in [0,|S|-2]}W(x+W_T)^{-1}\odot \phi_H(\beta)^{|S|-2}|=1.\ee As a result, since $|S|\geq 4$, comparing the values $x=0$ and  $x=1$  in \eqref{cupsache} shows that  $$(\frac{(|S|-1)|S|}{2}-1)\phi_H(\beta)=(\frac{(|S|-1)|S|}{2}-3)\phi_H(\beta),$$ whence $2\phi_H(\beta)=0$. However, since $\supp(\phi_H(S))=\{0,\phi_H(\beta)\}$ must generate $G/H$, this implies that $|G|=|G/H|\cdot |H|=2\cdot 2=4$, contradicting the assumption  $|G|\geq 5$. Thus \eqref{H-biggish} is established in all cases.

We can assume \be\label{S-upper-H} 2\leq |H|\leq \frac{|S|-1}{2},\ee else the desired conclusion $|W\odot S|\geq |S|$  follows from \eqref{H-biggish}.
We divide the remainder of the case into several subcases.

\subsection*{Subcase 3.1:} $K=G$ and $|S|\geq |H|+|G/H|+1$.

In this case, we can apply the induction hypothesis to $\phi_H(S{T'}^{-1})$ to conclude that
 $$(WW_T^{-1})\odot \phi_H(S{T'}^{-1})=G/H.$$ Hence, since $W_T\odot {T'}$ is an $H$-coset, it follows that
 $G=(WW_T^{-1})\odot (S{T'}^{-1})+W_T\odot T'\subseteq W\odot S$, as desired.

\subsection*{Subcase 3.2:} $|S|\leq |H|+|K/H|-1+\epsilon$, where $\epsilon=0$ if $|K/H|\geq 3$ and $\epsilon=1$ if $|K/H|\leq 2$.

In this case, we can apply the induction hypothesis to $WW_T^{-1}\odot \phi_H(S{T'}^{-1})$, recall that $W_T\odot T'$ is an $H$-coset, and use the bounds given by \eqref{S-upper-H} to conclude  that \be\label{fisht}|W\odot S|\geq |H|(|S|-|T'|)=|H|(|S|-|H|)\geq \min\{2|S|-4,\; \frac{|S|^2-1}{4}\}.\ee If the theorem fails for $S$, then  $|W\odot S|\leq |S|-1$, which combined with \eqref{fisht} yields the contradiction $|S|\leq 3$.


\subsection*{Subcase 3.3:} $|S|=|H|+|K/H|$.

In view of Subcase 3.2, we can assume $|K/H|\geq 3$, whence $|K|\geq 3|H|\geq 6$.
Applying the induction hypothesis to $WW_T^{-1}\odot \phi_H(S{T'}^{-1})$ and recalling that $W_T\odot T'$ is an $H$-coset, we conclude  that \be\label{fishs}|W\odot S|\geq |H|(|K/H|-1)=|K|-|H|.\ee If the theorem fails for $S$, then  $|W\odot S|\leq |S|-1=|H|+|K/H|-1$, which combined with \eqref{fishs} yields $$|K|\leq 2|H|+|K/H|-1.$$ However, in view of $2\leq |H|\leq \frac{|K|}{3}$, the above is only possible if $|K|=6$ and $|H|=2$. In this case, equality must hold in \eqref{fishs}, which is only possible (in view of $|K/H|=3$ and the characterization given by (ii)) if the $3$ terms of $\phi_H(S{T'}^{-1})$ are the $3$ distinct elements of some cardinality $3$ coset $\phi_H(\beta)+K/H$, where $\beta\in G$. Let $K/H=\{0,\phi_H(g),2\phi_H(g)\}$, where  $\ord(\phi_H(g))=3$ and $g\in G$, so that $$\phi_H(S{T'}^{-1})=\phi_H(\beta)\phi_H(\beta+g)\phi_H(\beta+2g).$$
Since $3\equiv 1\mod 2$, we have $(0)(3)\odot T'=H$, while $$(1)(2)(4)\odot \phi_H(S{T'}^{-1})=(1)(2)(4)\odot \phi_H(\beta)\phi_H(\beta+g)\phi_H(\beta+2g)=7\phi_H(\beta)+\{0,\phi_H(g),2\phi_H(g)\}$$ is a full $K/H$-coset, whence  $$7\beta+K=(0)(3)\odot T'+(1)(2)(4)\odot S{T'}^{-1}\subseteq W\odot S.$$ Thus $|W\odot S|\geq |K|=6>|S|$, as desired, which completes the subcase.

\medskip

Observe that Subcases 3.1--3.3 cover all possibilities when $K=G$. Thus it remains to consider the case when $K<G$ is proper, in which case \eqref{KG} shows $|H|=2$. Note that the following subcase covers all remaining possibilities.

\subsection*{Subcase 3.4:} $K<G$ is proper and  $|S|\geq |H|+|K/H|+1=|K/H|+3$.

In view of \eqref{KG}, we conclude there must be precisely $2$ terms of $S$ from $H$ for this subcase, else $T\neq T'$ and $K=G$ follows, contrary to subcase hypothesis.

Suppose $|S|\geq |H|+2|K/H|+1=|K|+3$. Then $|S{T'}^{-1}|\geq 2|K/H|+1=|K|+1\geq 3$. Recall that all terms of $S{T'}^{-1}$ are from the $K$-coset $\alpha+K$. Thus $\la\supp(S{T'}^{-1})\ra_*\leq K<G$. Hence, if $\la\supp(S{T'}^{-1})\ra_*$ is nontrivial, then, in view of $|S{T'}^{-1}|\geq |K|+1\geq 3$, we see that the hypotheses of Case 3 but not Subcase 3.4 hold using $S{T'}^{-1}$ and $\la\supp(S{T'}^{-1})\ra_*$ in place of $T$ and $H$, whence one of the previous subcases can be applied to complete the case. On the other hand, if $\la\supp(S{T'}^{-1})\ra_*$ is trivial, say w.l.o.g.~$S{T'}^{-1}=\alpha^{|S|-2}$, then we can translate $S$ so that $S=0^{|S|-2}xy$ and apply Case 2 to complete the subcase. So we may instead assume \be\label{almost}|S|\leq |K|+2.\ee

Since $|S{T'}^{-1}|=|S|-|H|\geq |K/H|+1$ holds by hypothesis, we can apply the induction hypothesis to $WW_T^{-1}\odot \phi_H(S{T'}^{-1})$ and recall that $W_T\odot T'$ is an $H$-coset to thereby conclude that \be\label{wiw}|W\odot S|\geq |K|.\ee If the theorem fails for $S$, then we must have $|W\odot S|\leq |S|-1$, which, in view of \eqref{almost} and \eqref{wiw}, is only possible if \be\label{lamset}2|K/H|+1=|K|+1\leq |S|\leq |K|+2.\ee
From \eqref{S-upper-H}, we have $|S|\geq 2|H|+1$, which combined with \eqref{lamset} implies that $|K/H|\geq 2$. 

Recall that $\supp(S{T'}^{-1})\subseteq \alpha+K$. Since $|K/H|\geq 2$, we infer from \eqref{lamset} that $|\phi_H(S{T'}^{-1})|\geq |K/H|+1$, whence applying the induction hypothesis to $\phi_H(S{T'}^{-1})$  shows that there exists a subsequence $R\mid S{T'}^{-1}$ with $|R|=|K/H|$ such that  $W'\odot \phi_H(R)$ is a $K/H$-coset for any sequence $W'$ consisting of $|K/H|$ consecutive integers.

Recall that $|K|\geq |H|\geq 2$. Thus, if $|W\odot S|\geq 2|K|$, then combining this with \eqref{almost} shows that $|W\odot S|\geq |S|$, as desired. Therefore we conclude that \be\label{jih}|W\odot S|<2|K|.\ee In view of the subcase hypothesis, $S{T'}^{-1}R^{-1}$ is a nontrivial sequence, so we may find some $g\in \supp(S{T'}^{-1}R^{-1})$. Since $(0)(1)\odot T'=H$ and $(2)(3)\cdot\ldots (|K/H|+1)\odot \phi_H(R)$ is a $K/H$-coset, we conclude that $$(0)(1)\cdot\ldots\cdot (|K/H|+2)\odot T'Rg$$ contains the full $K$-coset \be\label{coset1}\left(\frac{(|K/H|+1)(|K/H|+2)}{2}-1\right)\alpha+(|K/H|+2)g+K.\ee Likewise, since
$(1)(2)\odot T'=H$ and $(3)(4)\cdot\ldots (|K/H|+2)\odot \phi_H(R)$ is a $K/H$-coset, we conclude that $$(0)(1)\cdot\ldots\cdot (|K/H|+2)\odot T'Rg$$ also contains the full $K$-coset \be\label{coset2}\left(\frac{(|K/H|+2)(|K/H|+3)}{2}-3\right)\alpha+K.\ee As all terms of $S{T'}^{-1}$ are from $\alpha+K$, we have $\phi_K(\alpha)=\phi_K(g)$, while in view of $|W\odot W|<2|K|$, both $K$-cosets given in \eqref{coset1} and \eqref{coset2} must be equal; which implies  $2\phi_K(\alpha)=0$. As a result, we derive from \eqref{spaff} and $K<G$ that $|G/K|=2$.

If $|K/H|\leq 2$, then $|G|=|G/K||K/H|\leq 2\cdot 2=4$, contrary to assumption. Therefore we now conclude that $|K/H|\geq 3$. Next observe that $$(0)(2)\odot T'+(1)(3)(4)\cdot\ldots\cdot (|K/H|+2)\odot Rg\subseteq \left(\frac{(|K/H|+2)(|K/H|+3)}{2}-2\right)\alpha+K,$$  which is a $K$-coset disjoint from that of \eqref{coset2}. Consequently, \be\label{sclose}|W\odot S|\geq |K|+|(0)(2)\odot T'+(1)(3)(4)\cdot\ldots\cdot (|K/H|+2)\odot Rg|.\ee However, $(3)(4)\cdot\ldots\cdot (|K/H|+2)\odot \phi_H(R)$ is a full $K/H$-coset (as previously derived by use  of the induction hypothesis to define $R$), which readily implies that  $$|(0)(2)\odot T'+(1)(3)(4)\cdot\ldots\cdot (|K/H|+2)\odot Rg|\geq |K/H|\geq 3.$$ Combined with \eqref{sclose} and \eqref{lamset}, we conclude that $|W\odot S|\geq |K|+3\geq |S|+1$, as desired. This completes the final subcase of Case 1.
 For the remainder of the arguments, we return to considering $S$ translated so that $\vp_0(S)=\h(S)$.

\subsection*{Case 4:} $\frac13(|S|+2)\leq\h(S)\leq |S|-3$.

Note that the case hypothesis implies $|S|\geq 6$.
If $g\in \supp(S)$ is nonzero with $d:=\ord(g)\leq \lceil \frac13(|S|+2)\rceil$, then $0^dg\in \Fc(G)$ is a subsequence of $S$ with length $|0^dg|=d+1=|\langle g\rangle|+1\leq |S|\leq |G|$; moreover, $\langle \supp(0^dg)\rangle_*$ is equal to the proper (since the previous inequality implies $d<|G|$), nontrivial subgroup $\langle g\rangle$. Consequently,  Case 3 can be invoked to complete the proof. Therefore we instead conclude that \be\label{high-order}\ord(g)\geq \lceil \frac13(|S|+2)\rceil+1\quad \mbox{ for all nonzero } \;g\in \supp(S).\ee

Since $\vp_0(S)\leq |S|-3$, choose some nonzero $x\in \supp(S)$. In view of Case 1, we have $|\supp(S)|\geq 3$, whence there must be some other nonzero $y\in \supp(S)$ with $x\neq y$. If, for every such nonzero $y\in \supp(S)$ with $x\neq y$, we have $y\in\la x\ra$, then $\la x\ra=\la \supp(S)\ra=G$. Otherwise, we can find some nonzero $y\in \supp(S)$ with $x\neq y$ and $\la x,y\ra>\la x\ra$. As a result, choosing the nonzero $y\in \supp(S)\setminus\{0,x\}$ appropriately and setting $K_1=\langle x,y\rangle$, we obtain \be \label{K-size}|K_1|=|\langle x,y\rangle|\geq \min\{|G|,\,2\ord(x)\}\geq \min\{|G|,\,2\lceil \frac13(|S|+2)\rceil+2\},\ee where the latter bound follows from \eqref{high-order}.

Let $R_1=0^{\lceil\frac13(|S|+2)\rceil-1}xy$. In view of the case hypothesis, we see that $R_1\mid S$ with $0\in \supp(SR_1^{-1})$. Let $R_2=SR_1^{-1}$, so that, in view of $|S|\geq 6$ and the previous observation, we have \be\label{zeroisin}0\in \supp(R_1)\cap \supp(R_2).\ee Let $K_2=\la \supp(R_2)\ra_*=\la \supp(R_2)\ra$. In view of \eqref{zeroisin}, we also have $K_1=\la \supp(R_1)\ra_*=\la \supp(R_1)\ra$. Observe that \be\label{K+KisG}K_1+K_2=\la \supp(S)\ra=G.\ee
From the case hypothesis $\vp_0(S)\leq |S|-3$ and \eqref{zeroisin}, we see that $|\supp(R_2)|\geq 2$, whence $K_2$ is nontrivial. If $|K_2|\leq |R_2|-1\leq |S|-1\leq |G|-1$, then $K_2$ will be proper and $R_2$ will be a sequence of length at least $|K_2|+1$ all of whose terms come from the coset $0+K_2$, whence Case 1 completes the proof. Therefore we can assume \be \label{k2-big} |K_2|\geq |R_2|=|S|-|R_1|=|S|-\lceil\frac13(|S|+2)\rceil-1=\lfloor\frac{2|S|-5}{3}\rfloor.\ee
From \eqref{high-order}, we also have \be\label{k2-big-smallvaluess}|K_2|\geq  \lceil \frac13(|S|+2)\rceil+1.\ee

Let $A_1=(0)(1)\cdot\ldots\cdot (|R_1|-1)\odot R_1$ and let
$A_2=(|R_1|)(|R_1|+1)\cdot\ldots\cdot (|S|-1)\odot R_2$.
In view of \autoref{lemma-generation}, we have
$\la A_1\ra_*=\la\supp(R_1)\ra_*=K_1$ and $\la A_2\ra_*=\la \supp(R_2)\ra_*=K_2$.
Also, from their definition, we have \be A_1+A_2\subseteq W\odot S\label{A+Abound}.\ee
Since $|\supp(R_2)|\geq 2$, it is readily deduced that $|A_2|\geq 2$.
In consequence, if $|A_1|\geq |G|-1$, then applying \autoref{pigeonhole-lemma} to
$A_1+A_2$ shows that $A_1+A_2=G$, which in view of \eqref{A+Abound} completes the proof.
Therefore we can assume $|A_1|\leq |G|-2$. Consequently, in view of
$\la\supp(R_1)\ra_*=K_1$ and \eqref{K-size}, applying  \autoref{key-lemma}(iii) to $R_1$
 results in \be|A_1|\geq 2|R_1|-2=2\lceil \frac13(|S|+2)\rceil.\label{A1isbig}\ee
 Since $|R_2|<|S|$, we can apply the induction hypothesis to $R_2$ to yield
 \be\label{A2isbig}|A_2|\geq \min\{|K_2|-1,\,|R_2|\}=\min\{|K_2|-1,\,\lfloor\frac{2|S|-5}{3}
 \rfloor\}\geq \lfloor\frac{2|S|-5}{3}\rfloor-1,\ee where the final inequality follows from
 \eqref{k2-big}.

If  $|A_1+A_2|\geq |A_1|+|A_2|-1$, then \eqref{A1isbig}, \eqref{A2isbig} and \eqref{A+Abound} together yield $$|W\odot S|\geq 2|R_1|-2+|R_2|-2=|S|+|R_1|-4=|S|+\lceil \frac13(|S|+2)\rceil-3,$$ which is at least $|S|$ for $|S|\geq 6$, as desired. So we can instead assume \be\label{kt-canbeapplied} |A_1+A_2|<|A_1|+|A_2|-1.\ee


Let $H=\mathsf H(A_1+A_2)$ be the maximal period of $A_1+A_2$. In view of \eqref{kt-canbeapplied} and Kneser's Theorem, it follows that $H$ is a proper (else $W\odot S=G$, as desired), nontrivial subgroup with \be\label{ktapplied} |\phi_H(A_1+A_2)|\geq |\phi_H(A_1)|+|\phi_H(A_2)|-1.\ee
We divide the remainder of the case into several subcases.

\subsection*{Subcase 4.1:} $|\phi_H(A_1)|=|\phi_H(A_2)|=1$.

In this case, $K_1=\langle A_1\rangle_*\leq H$ and $K_2=\langle A_2\ra_*\leq H$, whence $G=K_1+K_2\leq H$ follows from \eqref{K+KisG}, contradicting that $H<G$ is proper.



\subsection*{Subcase 4.2:} $|\phi_H(A_1)|\geq 2$ and $|\phi_H(A_2)|=1$.

In this case, $K_2=\langle A_2\rangle_*\leq H$ and $|A_1+A_2|\geq 2|H|\geq 2|K_2|$, which is at least $\frac43 |S|-\frac{14}{3}$ in view of \eqref{k2-big}. For $|S|\geq 12$, combing this with \eqref{A+Abound} implies $|W\odot S|\geq |A_1+A_2|>|S|-1$, as desired. For $|S|\leq 11$, we can use \eqref{k2-big-smallvaluess} and \eqref{A+Abound} to estimate $|W\odot S|\geq |A_1+A_2|\geq 2|K_2|\geq \frac23 |S|+\frac{10}{3}>|S|-1$, also as desired.



\subsection*{Subcase 4.3:} $|\phi_H(A_2)|\geq 2$.

In this case, \eqref{ktapplied} and \eqref{A+Abound} imply $$|W\odot S|\geq |A_1+A_2|\geq |A_1+H|+|A_2+H|-|H|\geq |A_1+H|+\frac12|A_2+H|\geq |A_1|+\frac12|A_2|.$$ Combined with \eqref{A1isbig} and \eqref{A2isbig}, we obtain $$|W\odot S|\geq 2\lceil \frac13(|S|+2)\rceil+\frac12(\lfloor\frac{2|S|-5}{3}\rfloor-1)> |S|-1,$$ as desired, which completes the last subcase of Case 4.

\subsection*{Case 5:} $\h(S)\leq \frac13(|S|+1)$.

Let $$\epsilon =\left\{
                 \begin{array}{ll}
                   1, & \hbox{if } |S|\equiv 2\mod 3 \\
                   0, & \hbox{else}
                 \end{array}
               \right.$$
and let $r=\lfloor\frac13(|S|+1)\rfloor$. Note $r\geq 1$ in view of $|S|\geq 4$. We assume by contradiction that $S$ fails to satisfy the theorem (solely for the statements of the properties below, which might not hold if $S$ satisfied the conditions of the theorem).

The assumption $\h(S)\leq \frac13(|S|+1)$ allows us to factorize the sequence $S$ into square-free subsequences in the following way (this is the basic construction for the existence of an $r$-setpartition; see \cite{egzII}):\begin{itemize}\item If $|S|\equiv 0\mod 3$, then $r=\frac13|S|$, $\epsilon=0$, and we can factorize $S=S_1\cdot\ldots\cdot S_r$ such that $|\supp(S_i)|=|S_i|=3$ for  all $i\in [1,r]$.
\item If $|S|\equiv 1\mod 3$, then $r=\frac13(|S|-1)$, $\epsilon=0$, and we can factorize $S=S_1\cdot\ldots\cdot S_rS_{r+1}$ such that $|\supp(S_i)|=|S_i|=3$ for  all $i\in [1,r]$ and $|\supp(S_{r+1})|=|S_{r+1}|=1$.
\item If $|S|\equiv 2\mod 3$, then $r=\frac13(|S|+1)$, $\epsilon=1$, and we can factorize $S=S_1\cdot\ldots\cdot S_r$ such that $|\supp(S_i)|=|S_i|=3$ for  all $i\in [1,r-1]$ and $|\supp(S_r)|=|S_r|=2$.\end{itemize}
Note $\epsilon$ counts the number of $S_i$ with length $2$ in the factorization.
For the purposes of the proof, we will refer to a factorization $S_1\cdot\ldots\cdot S_r$ (of $S$ or $SS_{r+1}^{-1}$) as \emph{well-balanced} if it satisfies the above criteria and also has $|\langle \supp(S_j)\rangle_*|\geq 5$ for any $S_j$ with $|S_j|\geq 3$. Let us show that such a factorization exists.

\medskip

Let $S_1\cdot\ldots\cdot S_r\mid S$ be a factorization satisfying the appropriate bulleted criteria above. We trivially have $|\langle \supp(S_j)\rangle_*|\geq 3$ for each $S_j$ with $|S_j|=|\supp(S_j)|=3$. If $|\langle \supp(S_j)\rangle_*|=4$, then the pigeonhole principle guarantees that there are distinct $x,\,y\in \supp(S_j)$ with $\ord(x-y)=2$, whence invoking Case 3 with $H=\langle x-y\rangle$ shows that the theorem holds for $S$, contrary to assumption. Therefore, we see that $|\langle \supp(S_j)\rangle_*|\geq 5$ or $|\langle \supp(S_j)\rangle_*|=3$ for each $S_j$ with $|S_j|=3$. Consider a factorization $S_1\cdot\ldots\cdot S_r\mid S$ satisfying the appropriate bulleted criteria so that the number of $S_j$ with $|S_j|=|\langle \supp(S_j)\rangle_*|=3$ is minimal. If by contradiction no well-balanced factorization exists, then there will be some $S_j$ with $|S_j|=|\langle \supp(S_j)\rangle_*|=3$. Thus $\supp(S_j)$ is a coset of the cardinality $3$ subgroup $H:=\langle \supp(S_j)\rangle_*$. In view of $|S|\geq 4$, there is some $S_k$ with $k\in [1,r+1]$, $k\neq j$, and $k=r+1$ only if $|S|=4\equiv 1\mod 3$. If $\supp(S_k)$ and $\supp(S_j)$ share a common element, then there will be $4$ terms of $S$ from the same cardinality three $H$-coset, whence invoking Case 3 shows that the theorem holds for $S$, contrary to assumption.  Therefore we may instead assume that $\supp(S_k)$ and $\supp(S_j)$ are disjoint. Thus if we swap any term $x$ from $S_j$ for a term $y$ from  $S_k$ and let $S'_j=S_jx^{-1}y$ and $S'_k=S_ky^{-1}x$ denote the resulting sequences, then \autoref{punctured-lem} guarantees that $\supp(S'_j)$ cannot be periodic. In particular, $\supp(S'_j)$ is not a coset of  a cardinality $3$ subgroup. If $\supp(S'_k)$ is also not a coset of a cardinality three subgroup, then set $S''_j=S_j$ and $S''_k=S'_k$. On the other hand, if $\supp(S'_k)$ is a coset of a cardinality $3$ subgroup, then \autoref{punctured-lem} again shows that $S''_k:=S'_k{y'}^{-1}y$ is not periodic, and thus not coset of cardinality $3$ subgroup, where $y'$ is any element from $\supp(S_k)$ distinct from $y$. Moreover, we also have $S''_j:=S'_jy^{-1}y'=S_jx^{-1}y'$ not being a coset of a cardinality three subgroup (by repeating the arguments used to show this for $S'_j$ only using $y'$ instead of $y$). However, now the factorization $S_1\cdot\ldots\cdot S_rS_j^{-1}S_k^{-1}S''_jS''_k$ satisfies the appropriate bulleted condition and also has at least one less $S_j$ with $|S_j|=|\langle \supp(S_j)\rangle_*|=3$, contradicting the assumed minimality assumption. This shows that a well-balanced factorization $S_1\cdot\ldots\cdot S_r$ exists.

\medskip

For the moment, let $S_1\cdot\ldots\cdot S_r\mid S$ be an arbitrary well-balanced factorization. Let $W=W_1\cdot \ldots\cdot W_r$ be a factorization of $W$ with $|W_i|=|S_i|$ for all $i\in [1,r]$ such that each $W_i$ is a sequence of consecutive integers. Note we can apply \autoref{special-lemma} to each $S_j$ with $|S_j|=3$ since the definition of a well-balanced factorization ensures that $|\langle \supp(S_j)\rangle_*|\geq 5$ while we have $\ord(x-y)\geq 3$ for all distinct $x,\,y\in \supp(S_j)$, else Case 3 applied with $H=\langle x-y\rangle$ shows that the theorem holds for $S$, contrary to assumption. For each $S_j$ with $|S_j|=3$, let $A_j\subseteq W_j\odot S_j$ be the resulting subset with \be\label{condition-list}|A_j|=4,\quad \langle A_j\rangle_*=\langle \supp(S_j)\rangle_*,\quad \mbox{ and either }\quad \langle A_j\rangle_*\cong C_6 \quad \mbox{ or } \quad  |\mathsf H(A_j)|\neq 2.\ee For any $S_j$ with $|S_j|\neq 3$, let $A_j=W_j\odot S_j$. If $|S|\not\equiv 1\mod 3$, set $A_{r+1}=\{0\}$. Note that $|A_{r+1}|=1$ (regardless of the value of $|S|$ modulo $3$) and that $|A_r|=2$ when $|S_r|=2$.
We also have  \be\label{stunt} \Sum{i=1}{r+1}A_i\subseteq W\odot S.\ee
For the purposes of the proof, we will refer to a setpartition $\mathscr A=A_1\cdot\ldots\cdot A_rA_{r+1}$ obtained as above from a well-balanced factorization $S_1\cdot\ldots\cdot S_r\mid S$ as a \emph{well-balanced} setpartition.

Our plan is to show that a well-balanced setpartition with maximal cardinality sumset has $|\Sum{i=1}{r+1}A_i|\geq |S|$, which in view of \eqref{stunt} will yield the concluding contradiction $|W\odot S|\geq |S|$. To do this, we must first establish some properties that any well balanced setpartition has. We begin with the following.

\subsection*{Property 1:} If $A_1\cdot\ldots\cdot A_rA_{r+1}$ is a well-balanced setpartition
and \be\label{toosmall} |\Summ{i\in I}A_i|<\Summ{i\in I}|A_i|-|I|+1,\ee where
$I\subseteq [1,r]$ is a nonempty subset, then $|\mathsf H(\Summ{i\in I}A_i)|\geq 5$.

Let $H=\mathsf H(\Summ{i\in I}A_i)$ and suppose by contradiction that $|H|\leq 4$. In view of \eqref{toosmall} and Kneser's Theorem, we know $|H|\geq 2$ with
$$|\Summ{i\in I}A_i|\geq \Summ{i\in I}|A_I|-|I|+1-(|I|-1)(|H|-1)+\rho,$$ where
$\rho=\Summ{i\in I}(|A_i+H|-|A_i|)$ denotes the number of $H$-holes in the $A_i$ with $i\in I$. In particular, \be\label{holebound}\rho<(|I|-1)(|H|-1).\ee

Suppose $|H|\in \{3,4\}$. Now all but at most one $A_i$ with $i\in I\subseteq [1,r]$ has $|A_i|=4$. Since $|\langle A_i\rangle_*|=|\la \supp(S_i)\ra_*|\geq 5>|H|$ for such $A_i$, we know that each such $A_i$ intersects at least two $H$-cosets, whence $$|A_i+H|-|A_i|\geq 2|H|-4\geq |H|-1.$$ Thus $\rho=\Summ{i\in I}(|A_i+H|-|A_i|)\geq (|I|-1)(|H|-1)$, contradicting \eqref{holebound}. So we may instead assume $|H|=2$.

If $A_i$ is $H$-periodic with $|A_i|=2$, then \autoref{lemma-generation} implies that $S_i$ consists of $2$ distinct elements from the same cardinality $2$ $H$-coset, whence applying Case 3
shows that the theorem holds for $S$, contrary to assumption. Therefore only $A_i$ with $|A_i|=4$ can be $H$-periodic.

If at most one $A_i$ with $i\in I$ is $H$-periodic, then $|A_i+H|-|A_i|\geq 1=|H|-1$ will hold for all but at most one  $i\in I$, and we will again contradict \eqref{holebound}. Therefore there must be at least two $A_i$ with $i\in I$ that are $H$-periodic, and in view of the previous paragraph, we must have $|A_i|=4$ for each such $A_i$. However \eqref{condition-list} shows  this is only possible for $A_i$   if $\langle\supp(S_i)\rangle_*=\langle A_i\rangle_*\cong C_6$, in which case  $A_i$ is a cardinality $4$ subset of a coset of the cardinality $6$ subgroup $\langle A_i\rangle_*$.

Let $J\subseteq I$ be the subset of all those indices $i\in I$ such that $A_i$ is $H$-periodic. Since there are at least two $A_i$ with $i\in I$ and $A_i$ being $H$-periodic, as shown above, we have $|J|\geq 2$. By the argument of the previous paragraph, each $A_i$ with $i\in J$ has $\langle \phi_H(A_i)\rangle_*\cong C_3$. Thus, if  $\langle \phi_H(A_i)\rangle_*=\langle \phi_H(A_j)\rangle_*$ for distinct $i,\,j\in J$, then \autoref{pigeonhole-lemma} implies that $A_i+A_j$ is $\langle A_j\rangle_*$-periodic, contradicting that $H<\langle A_j\rangle_*$ is the maximal period of $\Summ{i\in I}A_i$. Therefore we may assume each $\langle \phi_H(A_i)\rangle_*$, for $i\in J$, is a distinct cardinality $3$ subgroup.
In consequence, we have \be\label{dewdrop}|\phi_H(A_i)+\phi_H(A_j)|=4\quad\mbox{ for }\;i,\,j\in J\;\mbox{ distinct}.\ee
Since $H$ is the maximal period of $\Summ{i\in I}A_i$ and $J\subseteq I$, it follows that $\Summ{i\in J}\phi_H(A_i)$ is aperiodic. Thus, pairing up the $\phi_H(A_j)$ with $j\in J$ into  $\lfloor\frac12|J|\rfloor$ pairs, applying the equality \eqref{dewdrop} to each pair, and then applying Kneser's Theorem to the aperiodic $\lceil\frac12|J|\rceil$-term sumset whose summands consist of the sumsets of each of the $\lfloor\frac12|J|\rfloor$ pairs along with the one unpaired set $\phi_H(A_i)$ with $i\in J$ (if $|J|$ is odd) yields the estimates
 \begin{align}\label{hopper}&|\Summ{i\in J}\phi_H(A_i)|\geq 4\left (\frac{|J|-1}{2}\right)+2-\frac{|J|+1}{2}+1=\frac32|J|+\frac12&\mbox{ if }\; |J|\;\mbox{ is odd,}\\
\nn &|\Summ{i\in J}\phi_H(A_i)|\geq 4\left (\frac{|J|}{2}\right)-\frac12|J|+1=\frac32|J|+1
&\mbox{ if }\; |J|\;\mbox{ is even.}\end{align}

For each $i\in I\setminus J\subseteq [1,r]$, we know $A_i$ is not $H$-periodic. As a result, if $i\in I\setminus J$ with $|A_i|=4$, then $|\phi_H(A_i)|\geq 3$, while if $i\in I\setminus J$ with $|A_i|=2$, then $|\phi_H(A_i)|=2$. Consequently, since $\Summ{i\in I}\phi_H(A_i)$ is aperiodic (as $H$ is the maximal period of $\Summ{i\in I}A_i$), Kneser's Theorem and \eqref{hopper} together imply
\be\label{onethree}|\Summ{i\in I}\phi_H(A_i)|\geq |\Summ{i\in J}\phi_H(A_i)|+\Summ{i\in I\setminus J}|\phi_H(A_i)|-(|I\setminus J|+1)+1\geq \frac32|J|+\frac12+2|I\setminus J|-\epsilon\geq \frac32|I|+\frac12-\epsilon.\ee  Since $\Summ{i\in I}A_i$ is $H$-periodic with $|H|=2$, \eqref{onethree} implies  $|\Summ{i\in I}A_i|\geq 3|I|+1-2\epsilon=\Summ{i\in I}|A_i|-|I|+1$, contradicting \eqref{toosmall} and completing the  proof of Property 1.

\medskip

\bigskip

Next, recalling the definition of $r$, we observe that $$\Sum{i=1}{r}|A_i|-r+1=4r-2\epsilon-r+1=3r-2\epsilon+1\geq |S|.$$ Consequently, in view of \eqref{stunt} and $W\odot S\neq G$, it follows that
\be\label{gogog} |\Sum{i=1}{r}A_i|<\min\{|G|,\;\Sum{i=1}{r}|A_i|-r+1\}.\ee Thus Property 1 ensures that
$H_1:=\mathsf H(\Sum{i=1}{r}A_i)$ has $|H_1|\geq 5$. Since $H_1$ must be a proper subgroup, it follows that $|G|$ is composite with $$|G|\geq 2|H_1|\geq 10.$$ Let $I_1\subseteq [1,r]$ denote all those indices $i\in [1,r]$ such that $|\phi_{H_1}(A_i)|=1$.
Our next goal is the following.

\subsection*{Property 2:} If $A_1\cdot\ldots\cdot A_rA_{r+1}$ is a well-balanced setpartition with $H_1=\mathsf H(\Sum{i=1}{r}A_i)$ and $I_1\subseteq [1,r]$ being the subset of all $i\in [1,r]$ with $|\phi_{H_1}(A_i)|=1$, then  $|I_1|\geq \lceil\frac13(|H_1|-2)\rceil+2$.

First let us handle the case when $|I_1|=r=\lfloor\frac{|S|+1}{3}\rfloor$. In this case, we need to show $|S|\geq |H_1|+5$, for which, in view of $|W\odot S|<|S|$, it suffices to show that $|W\odot S|\geq |H_1|+4$. Since $|\Sum{i=1}{r}A_i|\leq |W\odot S|<|S|$, we have the initial estimate $|S|\geq |H_1|+1$. However,  if $|W\odot S|=|H_1|$, then $\la\supp(S)\ra_*=\la W\odot S\ra_*=H_1<G$ follows from \autoref{lemma-generation}, contradicting the hypothesis $\la \supp(S)\ra_*=G$. Therefore we instead conclude that $|W\odot S|\geq |H_1|+1$, in turn implying \be\label{wheel}|S|\geq |W\odot S|+1\geq |H_1|+2\geq 7.\ee
Since $|I_1|=r$, we know that every $A_i$ with $i\in [1,r]$ is contained in an $H_1$-coset. Consequently, in view of \eqref{condition-list}, we see that each $S_i$ with $i\in [1,r]$ has all its terms from a single $H_1$-coset, say $\supp(S_i)\subseteq \alpha_i+H_1$. If it is the same $H_1$-coset for all $S_i$ with $i\in [1,r]$, then we will have at least $|S|-1\geq |H_1|+1$ terms from the same $H_1$-coset (in view of \eqref{wheel}), whence Case 3 shows that the theorem holds for $S$, contrary to assumption. Therefore we can instead assume $\alpha_j+H_1\neq \alpha_r+H_1$ for some $j\in [1,r-1]$.  Let $g_r\in \supp(S_r)$ and $g_j\in \supp(S_j)$ and define $A'_r=W_r\odot S_rg_r^{-1}g_j$ and $A'_j=W_j\odot S_jg_j^{-1}g_r$. For $i\in [1,r+1]\setminus \{r,\,j\}$, set $A'_i=A_i$. Then, since neither $S_rg_r^{-1}g_j$ nor $S_jg_j^{-1}g_r$ is contained in a single $H_1$-coset, it follows from \autoref{lemma-generation} that $|\phi_{H_1}(A'_j)|\geq 2$ and $|\phi_{H_1}(A'_r)|\geq 2$. In consequence, the subset $\Sum{i=1}{r+1}A'_i\subseteq W\odot S$ intersects at least two $H_1$-cosets, one of which must be disjoint from the $H_1$-coset that contained $\Sum{i=1}{r+1}A_i$.

If $r\geq 3$, then there will be some $A_i=A'_i$ with $i\in [1,r-1]\setminus \{j\}$, which will be a cardinality $4$ subset of a single $H_1$-coset, thus ensuring that every $H_1$-coset that intersects $\Sum{i=1}{r+1}A'_i$ must contain at least $4$ elements. As a result, if $r\geq 3$, then $|W\odot S|\geq |H_1|+4$, as desired. Therefore it remains to consider the case when $r\leq 2$ in order to finish the case when $|I_1|=r$. However, \eqref{wheel} shows that $r\leq 2$ is only possible if $|H_1|=5$, $|S|=7$, $r=2$ and $j=1$. In this case, $|S|\equiv 1\mod 3$, so that $S_{r+1}$ contains a term from $S$. Since $\alpha_1+H_1=\alpha_j+H_1\neq \alpha_r+H_1=\alpha_2+H_1$, we can w.l.o.g. assume $\alpha_2+H_1\neq \alpha_3+H_1$, where $\alpha_3$ is the single term from $S_3$. But now, defining $A''_1=A_1\subseteq (0)(1)(2)\odot S_1$, $A''_2=(3)(4)\odot S_2g_2^{-1}$ and $A''_3=(5)(6)\odot S_rg_2$, we can repeat the  arguments from the $r\geq 3$ case using the $A''_i$ instead of the $A'_i$ in order to conclude $|W\odot S|\geq |H_1|+4$ in this final remaining case as well. So, for the remainder of the proof of Property 2, we can now assume $|I_1|\leq r-1$.

From Kneser's Theorem, \eqref{stunt}, the definitions of $I_1$ and $r$, and the assumption $|W\odot S|<|S|$, we have \be\label{threeone}|S|-1\geq |\Sum{i=1}{r}A_i|\geq (r-|I_1|+1)|H_1|=(\lfloor \frac{|S|+1}{3}\rfloor-|I_1|+1)|H_1|,\ee from which we derive both
$$|I_1|\geq \lfloor \frac{|S|+1}{3}\rfloor+1-\frac{|S|-1}{|H_1|}\geq (|S|-1)\frac{|H_1|-3}{3|H_1|}+1$$
 and
$|S|\geq (e+1)|H_1|+1$, where $e:=r-|I_1|\geq 1$. Combining these inequalities yields
\be\nn|I_1|\geq (e+1)\frac{|H_1|}{3}-e.\ee Since $|H_1|\geq 5$, the above bound is minimized for small $e$. Thus, since $e\geq 1$, we obtain \be\label{taufiL}|I_1|\geq \lceil \frac23 |H_1|\rceil -1,\ee which is at least the desired bound $\lceil\frac13(|H_1|-2)\rceil+2$ except when $|H_1|=6$. In this case, we must have $|S|=2|H_1|+1=13$ with $e=1$, else the estimate \eqref{taufiL} will become strict, yielding the desired bound on $|I_1|$. Thus $r=4$.

Since $|S|=13\equiv 1\mod 3$, the set $S_{r+1}$ contains a term from $S$, say $\alpha_{r+1}$. In view of \eqref{condition-list} and the definition of $I_1$, we know each $\supp(S_i)$, for $i\in I_1$, is contained in a single $H_1$-coset. If this single $H_1$-coset is equal to $\alpha_{r+1}+H$ for each $i\in I_1$, then we will have $3|I_1|+1=10\geq |H_1|+1$ terms of $S$ from the same $H_1$-coset, whence invoking Case 3 shows that the theorem holds for $S$, contrary to assumption. Therefore there must be some $j\in I_1$ such that $\supp(S_j)\subseteq \alpha_j+H_1\neq \alpha_{r+1}+H_1$, say w.l.o.g. $j=r$.  Set $A'_{i}=A_i$ for $i\in [1,r-1]$, set $A'_r=(9)(10)\odot S_rg^{-1}$ and set $A'_{r+1}=(11)(12)\odot S_{r+1}g$, where $g\in \supp(S_r)$. Observe that $\phi_{H_1}(A_i)=\phi_{H_1}(A'_i)$ for $i\in [1,r-1]$ while $|\phi_{H_1}(A_r)|=|\phi_{H_1}(A'_r)|=1$. Consequently, $\Sum{i=1}{r}\phi_H(A'_i)$ is a translate of $\Sum{i=1}{r}\phi_H(A_i)$; in particular, $\Sum{i=1}{r}\phi_{H_1}(A'_i)$ is aperiodic in view of $H_1$ being the maximal period of $\Sum{i=1}{r}A_i$. However, since $\supp(S_{r+1}g)$ is not contained in a single $H_1$-coset, it follows from \autoref{lemma-generation} that $|\phi_{H_1}(A'_{r+1})|\geq 2$, whence, since $\Sum{i=1}{r}\phi_{H_1}(A'_i)$ is aperiodic, Kneser's Theorem implies that $$|\Sum{i=1}{r+1}\phi_{H_1}(A'_i)|>|\Sum{i=1}{r}\phi_{H_1}(A'_i)|=
|\Sum{i=1}{r}\phi_{H_1}(A_i)|=|\Sum{i=1}{r+1}\phi_{H_1}(A_i)|.$$ Thus $\Sum{i=1}{r+1}A'_i\subseteq W\odot S$ intersects some $H_1$-coset that is disjoint from  $\Sum{i=1}{r+1}A_i\subseteq W\odot S$, which combined with \eqref{threeone} and the definition of $e$ implies that $$|S|>|W\odot S|>|\Sum{i=1}{r+1}A_i|=|\Sum{i=1}{r}A_i|\geq (e+1)|H_1|=12,$$ yielding the contradiction $|S|\geq 14$.
Thus Property 2 is established in the final remaining case.

\subsection*{Property 3:} Let $A_1\cdot\ldots\cdot A_rA_{r+1}$ be a well-balanced setpartition, let $K\leq G$ be a subgroup, let $J\subseteq [1,r]$ be a subset of indices with  $|\phi_K(A_i)|=1$ and $|A_i|=4$ for all $i\in J$, let $L=\mathsf H(\Summ{i\in J}A_i)$, and let $I\subseteq J$ denote all those indices $i\in J$ with $|\phi_L(A_i)|=1$.  If $|J|\geq \lceil\frac13(|K|-2)\rceil$ and $5\leq |L|<|K|$, then $|I|\geq \lceil\frac13(|L|-2)\rceil+2$.

\medskip

Since $|\phi_K(A_i)|=1$ for all $i\in J$, each $A_i$ with $i\in J$ is contained in a single $K$-coset, whence $\Summ{i\in J}A_i$ is also contained in a single $K$-coset. Thus $L\leq K$, so that our hypothesis $|L|<|K|$ implies $|L|\leq \frac12|K|$. In particular, $$|K|\geq 2|L|\geq 10.$$
Suppose by contradiction that $|I|\leq \lceil\frac13(|L|-2)\rceil+1\leq \frac13|L|+1$. For each $i\in J\setminus I$, we have $|\phi_L(A_i)|\geq 2$. Thus, in view of $L\neq K$, Kneser's Theorem implies that $|J\setminus I|=|J|-|I|\leq |K/L|-2$. Combined with our assumption on the size of $|I|$ and the hypothesis for the size of $|J|$, we find that
\ber \label{ish2}\left\lceil\frac{|K|-2}{3}\right\rceil-|K/L|+2\leq |I|\leq \left\lceil\frac{|L|-2}{3}\right\rceil+1,\eer which implies $\frac13|K|\leq \frac13|L|+|K/L|-\frac13$, in turn yielding \be\label{weq}|K|\leq |L|+\frac{3|K|}{|L|}-1.\ee
Considering the right hand side of \eqref{weq} as a function of $|L|$, we find that its maximum will be obtained for a boundary value of $|L|$, i.e., for $|L|=5$ or $|L|=\frac12|K|$. If $|L|=\frac 12 |K|$, we obtain $|K|\leq \frac{1}{2}|K|+5$, and if $|L|=5$, we obtain $|K|\leq \frac35 |K|+4$. In view of $|K|\geq 10$, both of these inequalities can only hold for $|K|=10$ with $|L|=5$
(in view of $|L|\geq 5$).
 However, for these values, we see that \eqref{ish2} instead implies $3-2+2\leq 2$, a contradiction. Thus Property 3 is established.

\bigskip

With the above three properties established for an arbitrary well-balanced setpartition $\mathscr A=A_1\cdot\ldots\cdot A_rA_{r+1}$, we now proceed to complete the proof by considering a well-balanced setpartition satisfying an iterated list of extremal conditions. The argument that follows is a simple variation of the basic strategy used to proof the Partition Theorem \cite{ccd}.
During the course of the construction of $\mathscr A$, we will at times declare certain quantities fixed, by which we mean that any additional assumption on $\mathscr A$ is always subject to all previously fixed quantities being maintained in their current state.


We begin by setting $J_1=[1,r]$, fixing $S_{r+1}$, and assuming our well-balanced setpartition $A_1\cdot\ldots\cdot A_rA_{r+1}$ has maximal cardinality sumset $|\Summ{i\in J_1}A_i|<|S|\leq |G|$ (in view of $|W\odot S|<|S|$). Fix $\Summ{i\in J_1}A_i$ up to translation. Let $H_1=\mathsf H(\Summ{i\in J_1}A_i)$ and $I_1$ be as defined above Property 2.

Next assume that $|I_1|$ is minimal (subject to all prior fixed quantities and extremal assumptions).
We showed above that  $H_1=\mathsf H(\Sum{i=1}{r}A_i)$ has $|H_1|\geq 5$, while Property 2 ensures that $|I_1|\geq \lceil\frac13(|H_1|-2)\rceil+2$.
We have $\langle A_i\rangle_*\subseteq H_1$ for all $i\in I_1$, whence \eqref{condition-list} ensures that $\langle \supp(S_i)\rangle_*\subseteq H_1$ for all $i\in I_1$. Thus each $\supp(S_i)$, for $i\in I_1$, is contained in some $H_1$-coset.
If it is the same $H_1$-coset for every $i\in I_1$, then we will have at least $3|I_1|-\epsilon\geq 3(\frac13(|H_1|-2)+2)-\epsilon\geq |H_1|+1$ terms of $S$ all from the same $H_1$-coset, whence Case 3 applied using the group $\langle \supp(\prod_{i\in I_1}S_i)\rangle_*\leq H_1<G$ shows that the theorem holds for $S$, contrary to assumption.
Therefore we may instead assume that there are distinct $k_1,\,k'_1\in I_1$ with $\supp(S_{k_1})$ and $\supp(S_{k'_1})$ contained in distinct $H_1$-cosets; moreover, if $|A_j|=2$ for some $j\in I_1$, then we can additionally assume $j\in \{k_1,\,k'_1\}$. Let $J_2=I_1\setminus \{k_1,\,k'_1\}$. Note $|A_i|=4$ for all $i\in J_2$.

Fix $S_i$ for all $i\in [1,r]\setminus J_2$, next assume that $|\Summ{i\in J_2}A_i|$ is maximal subject to all prior extremal assumptions still holding, and then fix $\Summ{i\in J_2}A_i$ up to translation.
In view of $|J_2|=|I_1|-2\geq \lceil\frac13(|H_1|-2)\rceil$ and $|H_1|\geq 5$, we see that $|J_2|$ is nonempty. Moreover, we have
\be\label{run} \Summ{i\in J_2}|A_i|-|J_2|+1=3|J_2|+1\geq |H_1|-1.\ee

Let us next show that $|\Summ{i\in J_2}A_i|<|H_1|-1$. Suppose this is not the case: $|\Summ{i\in J_2}A_i|\geq |H_1|-1$. Now $\supp(S_{k_1})$ and $\supp(S_{k'_1})$ are contained in disjoint $H_1$-cosets. Consequently, if we can swap a term between $S_{k_1}$ and $S_{k'_1}$ with the result giving a well-balanced setpartition satisfying all extremal assumptions coming before the assumption on $|\Summ{i\in J_2}A_i|$, then we will have contradicted the minimality of $|I_1|$. We proceed to do so.

Let $x\in \supp(S_{k_1})$ and let $y\in \supp(S_{k'_1})$.
If swapping the terms $x$ and $y$ does not result in a well-balanced factorization, then w.l.o.g.~we must have  $|S_{k_1}|=3$ with $\supp(S_{k_1}x^{-1}y)$ a coset of a cardinality $3$ subgroup (as argued in the existence of a well-balanced setpartition). However, in view of \autoref{punctured-lem}, this means that $\supp(S_{k_1}x^{-1}y')$ is not periodic, and thus not a coset of cardinality $3$ subgroup, for all other $y'\in \supp(S_{k'_1}y^{-1})$. Moreover, if $|\supp(S_{k'_1})|=3$, then \autoref{punctured-lem} also ensures that $S_{k'_1}x{y'}^{-1}$ cannot be a coset of a cardinality $3$ subgroup for both remaining terms $y'\in \supp(S_{k'_1}y^{-1})$. Thus, for any $x\in \supp(S_{k_1})$, we can find a $y\in \supp(S_{k'_1})$ such that swapping $x$ for $y$ results in a well-balanced factorization, thus inducing a well-balanced setpartition  where $A'_{k_1}\subseteq W_{k_1}\odot (A_{k_1}x^{-1}y)$ and $A'_{k'_1}\subseteq W_{k'_1}\odot(A_{k'_1}y^{-1}x)$ are obtained via \autoref{special-lemma} and have replaced $A_{k_1}$ and $A_{k'_1}$. Furthermore, either $|A'_{k_1}|=4$ or $|A'_{k'_1}|=4$, say $|A'_{k_1}|=4$, and then the construction of  $A'_{k_1}$ given by \autoref{special-lemma} allows us to assume there is a $2$ element subset of $A'_{k_1}$ contained in an $H_1$-coset.

Since $|\Summ{i\in J_2}A_i|\geq |H_1|-1$, \autoref{pigeonhole-lemma} implies that $\Summ{i\in I_1}A_i$ was a full $H_1$-coset (it cannot be larger as all sets $A_i$ with $i\in J_2\subseteq I_1$ are each themselves contained in an $H_1$-coset). However, since $A'_{k_1}$ still contains two elements from an $H_1$-coset, \autoref{pigeonhole-lemma} also ensures that $\Summ{i\in J_2}A_i+A_{k'_1}+A_{k'_2}$ contains a translate of this $H_1$-coset. Thus an appropriate translate of the sumset of the new setpartition contains all elements of $\Sum{i=1}{r}A_i$, whence the maximality of $|\Sum{i=1}{r}A_i|$ ensures that the sumset has not changed up to translation. Hence, since there are two less sets contained in a single $H_1$-coset in the new setpartition, we see that we have contradicted the minimality of $|I_1|$.
So we instead conclude that $|\Summ{i\in J_2}A_i|<|H_1|-1$, as claimed,  which, in view of \eqref{run}, implies that \be\label{wallyw}|\Summ{i\in J_2}A_i|<\min\{|H_1|,\,\Summ{i\in J_2}|A_i|-|J_2|+1\}.\ee

In view of \eqref{wallyw} and Property 1, we see that $H_2:=\mathsf (\Summ{i\in J_2}A_i)$ has $5\leq |H_2|<|H_1|$. Let $I_2\subseteq J_2$ be all those indices $i\in J_2$ with $|\phi_{H_2}(A_i)|=1$.  Assume $|I_2|$ is minimal (subject to all prior fixed quantities and extremal assumptions). Since $|J_2|=|I_1|-2\geq \lceil\frac13(|H_1|-2)\rceil$, we can apply Property 3 (with $L=H_2$ and $K=H_1$) to conclude  $|I_2|\geq \lceil\frac13(|H_2|-2)\rceil+2$. As before, all terms $A_i$ with $i\in I_2$ are contained in a single $H_2$-coset but not all in the same $H_2$-coset, else applying Case 3 shows that the theorem holds for $S$, contrary to assumption. This allows us to find $k_2,\,k'_2\in I_2$ such that $A_{k_2}$ and $A_{k'_2}$ are contained in disjoint $H_2$-cosets. Set $J_3=I_2\setminus \{k_2,\,k'_2\}$. Now fix all  $S_i$ for all $i\in [1,r]\setminus J_3$, next  assume that $|\Summ{i\in J_3}A_i|$ is maximal subject to all prior extremal assumptions still holding, and then fix $\Summ{i\in J_3}A_i$ up to translation. Repeating the above arguments, we again find that
$$|\Summ{i\in J_3}A_i|<\min\{|H_2|,\;\Summ{i\in J_3}|A_i|-|J_3|+1\}.$$ Thus Property 1 implies that  $H_3:=\mathsf (\Summ{i\in J_2}A_i)$ has $5\leq |H_3|<|H_2|$. Iterating the arguments of this paragraph, we obtain an infinite chain of subgroups $\infty>|G|>|H_1|>|H_2|>|H_3|>\ldots$, which is clearly impossible.
This contradiction completes the proof. (Essentially, the only way the above process terminates after a finite number of steps is when we find enough elements from the same proper coset, whence Case 3 shows that the theorem holds for $S$.)
\end{proof}

\section{Distinct Solutions to a Linear Congruence}
\label{4}

Let $r\in [2,n]$ and let $\alpha,a_1,\ldots,a_r\in\Z$. For each $x\in \Z$, we let $\overline{x}\in C_n$ denote $x$ reduced modulo $n$. Consider the linear congruence
$$a_1x_1+\ldots+a_rx_r\equiv\alpha\mod n.$$ Since the $a_i$ are allowed to be zero, there is no loss of generality to assume $r=n$ when studying the above congruence, in which case we have
\begin{equation}
\label{problem2}
a_1x_1+\ldots+a_nx_n\equiv\alpha\mod n.
\end{equation}
It is a simple and well-known result that there is a solution $(x_1,\ldots,x_n)\in \Z^n$ to \eqref{problem2}  precisely when $\alpha\in \gcd(a_1,\ldots,a_n,n)\Z$. It is less immediate when  a solution $(x_1,\ldots,x_n)$ with all $x_i$ distinct modulo $n$ exists. However,
noting  that the elements $a_1x_1+\ldots+a_nx_n$ having the $x_i$ distinct modulo $n$, when considered modulo $n$, are precisely the elements of $W\odot S$, where $W=0(1)\cdot\ldots (n-1)\in \Fc(\Z)$ and $S=\overline{a_1}\cdot\overline{a_2}\cdot\ldots\cdot \overline{a_n}\in \Fc(C_n)$, we then see that there existing a solution to \eqref{problem2} is equivalent to asking whether $\overline{\alpha}\in W\odot S$. If $n\geq 3$, then our main result \autoref{main} shows that $\overline{\alpha}\in W\odot S$ typically holds precisely when  \be\label{characterizaiton-general}\alpha\in \frac{(n-1)n}{2}a_1+\gcd(a_2-a_1,a_3-a_1,\ldots,a_n-a_1,n)\Z, \ee the only exception being when, for some distinct $j,\,k,\, l\in [1,n]$, we have $a_j-a_l\equiv-a_k+a_l\mod n$, \ $\gcd(a_j-a_l,n)=1$, and  $a_i\equiv a_l\mod n$ for all $i\in [1,n]\setminus\{j,k\}$, in which case $\overline{\alpha}\in W\odot S$ instead holds  precisely when \be\label{characterization-special}\alpha\in \frac{(n-1)n}{2}a_l+(\Z\setminus n\Z).\ee  Thus \autoref{main} characterizes when a solution $(x_1,\ldots,x_n)\in \Z^n$ to \eqref{problem2} exists having all $x_i$ distinct modulo $n$.

When $\alpha=1$, the congruence \eqref{problem2} becomes
\begin{equation}
\label{problem}
a_1x_1+\ldots+a_nx_n\equiv 1\mod n.
\end{equation}
Fairly recently, in \cite{pon10}, solutions to \eqref{problem} with all $x_i$ distinct modulo $n$ were constructed under the assumption that $\gcd(a_1,n)=\ldots=\gcd(a_k,n)=1$ and $a_{k+1}=\ldots=a_n=0$ for some $k<\varphi(n)$, where $\varphi(\cdot)$ denotes the Euler totient function. Additionally, \cite[Theorem 2]{pon10} proves the special case of \autoref{theo2} when $n$ is prime, and \autoref{theo2} generalizes \cite[Conjecture 3]{pon10}.

When $n=2$, there are essentially only three possible choices for $(a_1,a_2)$, namely $(0,0)$, $(0,1)$, and $(1,1)$. For $(0,0)$, there is no solution $(x_1,x_2)$ to \eqref{problem} with the $x_i$ distinct modulo $2$; for  $(0,1)$, there is a solution $(x_1,x_2)$ to \eqref{problem2} with the $x_i$ distinct modulo $2$ for all $\alpha$; and for $(1,1)$, there is  a solution $(x_1,x_2)$ to \eqref{problem} with the $x_i$ distinct modulo $2$ but no such solution to \eqref{problem2} for $\alpha=0$. The following result gives some special instances of the characterization given by \eqref{characterizaiton-general} and \eqref{characterization-special} for $n\geq 3$.

The first corollary addresses the question of when every $\alpha\in \Z$ has a solution $(x_1,\ldots,x_n)$ to \eqref{problem2} with the $x_i$ distinct modulo $n$.

\begin{corollary}
\label{theo1}
Let $n\geq 3$ and let $a_1,\ldots,a_n\in\Z$.
\begin{enumerate}
\item\label{theo1.1} If, for some  distinct $j,\,k,\, l\in [1,n]$, we have $a_j-a_l\equiv-a_k+a_l\mod n$, \ $\gcd(a_j-a_l,n)=1$, and  $a_i\equiv a_l\mod n$ for all $i\in [1,n]\setminus\{j,k\}$,  then
there is a solution $(x_1,\ldots,x_n)\in \Z^n$ to \eqref{problem} with the $x_i$ distinct modulo $n$ but there is some $\alpha\neq 1$ for which there is no solution $(x_1,\ldots,x_n)$ to \eqref{problem2} with all the $x_i$ distinct modulo $n$.
\item\label{theo1.2} Otherwise, the following are equivalent.
\begin{enumerate}
\item For every $\alpha\in\Z$, there is a solution $(x_1,\ldots,x_n)$ to \eqref{problem2} with the $x_i$ distinct modulo $n$.
\item For some $i\in [1,n]$, $\gcd(a_1-a_i,\ldots,a_n-a_i,n)=1$.
\end{enumerate}
\end{enumerate}
\end{corollary}
\begin{proof}
Noting that $\gcd(a_1-a_i,\ldots,a_n-a_i,n)=\gcd(a_1-a_j,\ldots,a_n-a_j,n)$ for all $i,\,j\in [1,n]$, it follows that these are both simple consequences of \eqref{characterization-special} and \eqref{characterizaiton-general}.
%
\qedhere
\end{proof}


The next result addresses the question of when \eqref{problem} has a solution $(x_1,\ldots,x_n)$ with the $x_i$ distinct modulo $n$. We remark that the arguments used below for $\alpha=1$ would actually work for any $\alpha\in \Z$ with $\gcd(\alpha,n)=1$.

\begin{theorem}
\label{theo2}Let $n\geq 2$ and let $a_1,\ldots,a_n\in\Z$.
\begin{enumerate}
 \item If $n$ is odd or some $a_i$ is even, then \eqref{problem} has a solution $(x_1,\ldots,x_n)$ with the $x_i$ distinct modulo $n$ if and only if $\gcd(a_2-a_1,a_3-a_1,\ldots,a_n-a_1,n)=1$.
     \item If $n\equiv 0\mod 4$ and all $a_i$ are odd, then \eqref{problem} has no solution $(x_1,\ldots,x_n)$ with the $x_i$ distinct modulo $n$.
         \item If $n\equiv 2\mod 4$ and all $a_i$ are odd, then \eqref{problem} has a solution $(x_1,\ldots,x_n)$ with the $x_i$ distinct modulo $n$ if and only if $\gcd(a_2-a_1,a_3-a_1,\ldots,a_n-a_1,n)=2$.
\end{enumerate}
\end{theorem}

\begin{proof}

That the theorem holds for $n=2$ can be easily checked, so we assume $n\geq 3$.

1. If the $a_i$ satisfy the hypothesis of \autoref{theo1}.1, then $\gcd(a_2-a_1,a_3-a_1,\ldots,a_n-a_1,n)=1$ and \autoref{theo1}.1 shows that \eqref{problem} has a solution. Therefore assume the $a_i$ do not satisfy the hypothesis of \autoref{theo1}.1. If $n$ is odd, then $\frac{(n-1)n}{2}a_1\equiv 0\mod n$, whence \eqref{characterizaiton-general} shows that \eqref{problem} has a solution $(x_1,\ldots,x_n)$ with the $x_i$ distinct modulo $n$ if and only if $\gcd(a_2-a_1,a_3-a_1,\ldots,a_n-a_1,n)=1$. If some $a_i$ is even, then we may w.l.o.g. re-index so that $a_1$ is even, whence $\frac{(n-1)n}{2}a_1\equiv 0\mod n$ again holds, completing the proof as before.

2. Since the $a_i$ are odd and $n$ is even, we have $2\mid \gcd(a_2-a_1,a_3-a_1,\ldots,a_n-a_1,n)$, in which case the hypotheses of \autoref{theo1}.1 cannot hold for the $a_i$. Additionally, since $4\mid n$, we have $2\mid \frac{(n-1)n}{2}a_1$ as well, whence $$\frac{(n-1)n}{2}a_1+\gcd(a_2-a_1,a_3-a_1,\ldots,a_n-a_1,n)\Z\subseteq 2\Z$$ and thus cannot contain $1$. Hence \eqref{characterizaiton-general} shows that \eqref{problem} has no solution $(x_1,\ldots,x_n)$ with the $x_i$ distinct modulo $n$.

3. As was the case in part 2, we have $2\mid \gcd(a_2-a_1,a_3-a_1,\ldots,a_n-a_1,n)$, so that the hypotheses of \autoref{theo1}.1 cannot hold for the $a_i$. Since $n\equiv 2\mod 4$ and $a_1$ is odd, we have $\frac{(n-1)n}{2}a_1\equiv\frac{n}{2}\mod n$. Thus \eqref{characterizaiton-general} shows that \eqref{problem} has a solution $(x_1,\ldots,x_n)$ with the $x_i$ distinct modulo $n$ if and only if $\frac{n}{2}-1\in \gcd(a_2-a_1,a_3-a_1,\ldots,a_n-a_1,n)\Z$. This condition rephrases as $\gcd(a_2-a_1,a_3-a_1,\ldots,a_n-a_1,n)\mid (\frac{n}{2}-1)$, which further rephrases as
\be\label{whoosh}\gcd(\frac{a_2-a_1}{2},\frac{a_3-a_1}{2},\ldots,\frac{a_n-a_1}{2},\frac{n}{2})\left| \frac{n-2}{4}.\right.\ee

If there were a common factor $p\geq 2$ dividing both $x$ and $\frac{x-1}{2}$, where $x\in \Z^+$, then $py=\frac{x-1}{2}$ and $pz=x$ for some positive integers $y,\,z\in\Z$, whence $2yp+1=x=pz$ follows, implying $p(z-2y)=1$, which contradicts that $p\geq 2$. Thus the integers $x$ and $\frac{x-1}{2}$ can share no common factors. Applying this observation with $x=\frac{n}{2}$, we see that \eqref{whoosh} holds precisely when $\gcd(\frac{a_2-a_1}{2},\frac{a_3-a_1}{2},\ldots,\frac{a_n-a_1}{2},\frac{n}{2})=1$, which is equivalent to $\gcd(a_2-a_1,a_3-a_1,\ldots,a_n-a_1,n)=2$. This completes the final part of the theorem.
\end{proof}

\section{Consequences for Minimal Zero-Sum Sequences}

We briefly recall the structure of minimal zero-sum sequences of maximal length in groups of rank $2$.
The following result was first shown as a conditional result in \cite[Theorem 3.2]{propB-3}, but by \cite{MR1985670}, \cite{propB-4}, \cite{propB-5} and \cite{propB-2}, the condition is satisfied.

\begin{lemma}[cf. {\cite[Theorem 3.2]{propB-3}}]
\label{inverse}
Let $G$ be a finite abelian group of rank two, say $G\cong C_m\oplus C_{mn}$ with $m,\,n\in\N$ and $m\geq 2$. The  minimal zero-sum sequences of maximal length are of the following forms.
\begin{enumerate}
 \item $S=e_j^{\ord e_j-1}\prod_{i=1}^{\ord e_k}(x_ie_j+e_k)$, where $\lbrace e_1,e_2\rbrace$ is a basis of $G$ with $\ord e_2=mn$, $\lbrace j,k\rbrace=\lbrace 1,2\rbrace$, and $x_i\in\N_0$ with $\sum_{i=1}^{\ord e_k}x_i\equiv 1\mod\ord e_j$.
 \item $S=g_1^{sm-1}\prod_{i=1}^{(n+1-s)m}(x_ig_1+g_2)$, where $s\in [1,n]$, $\lbrace g_1,g_2\rbrace$ is a generating set of $G$ with $\ord g_2=mn$ and, in case $s\neq 1$, $mg_1=mg_2$ and $x_i\in\N_0$ with $\sum_{i=1}^{(n+1-s)m}x_i=m(n(n+1-s)-1)+1$.
\end{enumerate}
\end{lemma}

In the second case of \autoref{inverse}, the coefficients $x_i$ are determined by equations only. But in the first case of \autoref{inverse}, the coefficients $x_i$ are determined by a congruence. Now suppose we are in case 1 and let $G$ and $S$ be as in \autoref{inverse}. Then we may write $S$ in the form
\[
S=e_j^{\ord e_j-1}\prod_{i=1}^{l}(x_ie_j+e_k)^{a_i},
\]
where $\lbrace e_1,e_2\rbrace$ is a basis of $G$, $\ord e_1=m$, $\ord e_2=mn$, $\lbrace j,k\rbrace=\lbrace 1,2\rbrace$, $l\in [1,\ord e_j]$, $a_1,\ldots,a_l\in\N$ with $a_1+\ldots+a_l=\ord e_k$, $x_1,\ldots,x_l\in [0,\ord e_j-1]$ and all the $x_i$ are distinct.
Note $a_1+\ldots+a_l=\ord e_k$ with each $a_i\geq 1$ implies that $l\leq \ord e_k$. Thus the characterization given by \autoref{inverse}.1 easily implies that $|\supp(S)|\in [3,\min\{l,\ord e_j\}]=[3,m+1]$ (we cannot have $|\supp(S)|=2$, as then all $x_i$ from \autoref{inverse}.1 would be equal modulo $\ord e_j$, in which case the congruence $x_1+\ldots x_{\ord e_k}\equiv 1\mod \ord e_j$ could not hold).
\smallskip

Here, we consider $\ord e_j-1,a_1,\ldots,a_l$ as a multiplicity pattern of the elements arising in $S$. Thus two natural questions appear:
\begin{itemize}
\item Which multiplicity patterns can occur?
\item How big can the support of $S$ be?
\smallskip
\end{itemize}

We use the main result from \autoref{4} to answer these questions. In particular, we will show that any value of $[3,m+1]$ can be achieved for $|\supp(S)|$, apart from $m+1$ when $n=1$ and $m\geq 3$, which, at least in the case $n=1$, was originally shown in \cite[Proposition 5.8.5]{alfredbook}.
First we set $a_i=0$ for $i\in [l+1,\ord e_j]$, choose $x_{l+1},\ldots,x_{\ord e_j}\in [0,\ord e_j-1]$ such that all $x_i$ are distinct, and obtain
\begin{align}
\label{congruB}
a_1x_1+\ldots+a_{\ord e_j}x_{\ord e_j}&\equiv 1\mod {\ord e_j}\und\\
\label{equB}
a_1+\ldots+a_{\ord e_j}&={\ord e_k}.
\end{align}
Now there are three possible cases depending on $\ord{e_j}$ and $\ord{e_k}$.
\smallskip\\
\textbf{Case 1.}
$\ord e_j=\ord e_k$, i.e. $n=1$ and $\ord e_j=\ord e_k=m$.
Then if  equation \eqref{equB} is satisfied, we must have either $a_1=\ldots=a_{\ord e_j}=1$ or $a_{\ord e_j}=0$. Now we apply \autoref{theo2} and find that there is only a solution to \eqref{congruB} in the first case when $m=2$, whence $|\supp(S)|=m+1$ is only possible when $m=2$, and that, in the second case, there is a solution to \eqref{congruB} for all choices of $a_1,\ldots,a_l$ with $\gcd(a_1,\ldots,a_l,{\ord e_j})=1\mod n$, where $1<l<{\ord e_j}=m$. In particular, taking the sequence $1^{l-1}(\ord e_j-l+1)0^{\ord e_j-l}$ for $a_1a_2\cdot\ldots\cdot a_{\ord e_j}$, where $l\in [2,\ord e_j-1]$, shows that any value of $|\supp(S)|\in [3,\ord e_j]=[3,m]$ is possible.
\smallskip\\
\textbf{Case 2.}
$\ord e_k<\ord e_j$, i.e., $\ord e_k=m$ and $\ord e_j=mn\geq 4$ with $m,\,n\geq 2$.
Then \eqref{equB} forces $a_{\ord e_j}=0$. Again we apply \autoref{theo2} and find that there is a solution to \eqref{congruB} for all choices of $a_1,\ldots,a_l$ with $\gcd(a_1,\ldots,a_l,{\ord e_j})=1\mod n$, where $1<l\leq\ord e_k<\ord e_j$.
In particular, taking the sequence $1^{l-1}(\ord e_k-l+1)0^{\ord e_j-l}$ for $a_1a_2\cdot\ldots\cdot a_{\ord e_j}$, where $l\in [2,\ord e_k]\subset [2,\ord e_j]$, shows that any value of  $|\supp(S)|\in[3,\ord e_k+1]=[3,m+1]$ is possible.
\smallskip\\
\textbf{Case 3.}
$\ord e_j<\ord e_k$, i.e., $\ord e_j=m$ and $\ord e_k=mn$. If $m=2$, then \eqref{congruB} has a solution provided $a_1$ and $a_2$ are both odd.
For $m\geq 3$, we apply \autoref{theo2} and obtain the following. The condition
\be\label{stingpin}\gcd(a_2-a_1,a_3-a_1\ldots,a_{\ord e_j}-a_1,\ord e_j)\leq 2\ee must always be fulfilled if \eqref{congruB} is to have a solution. Moreover, if $m$ is odd or some $a_i$ is even, then we must also have the inequality in \eqref{stingpin} being strict, while if $4\mid m$ and all $a_i$ are odd, then no solution to \eqref{congruB} can be found. In particular,
taking the sequence $1^{l-1}(\ord e_k-l+1)0^{\ord e_j-l}$ for $a_1a_2\cdot\ldots\cdot a_{\ord e_j}$, where $l\in [2,\ord e_j-1]=[1,m-1]$, shows that any value of  $|\supp(S)|\in[3,m]$ is possible. For $m\geq 3$, taking the sequence $1^{m-2}(2)(mn-m)$ for $a_1a_2\cdot\ldots\cdot a_{\ord e_j}$ shows that the value  $|\supp(S)|=m+1$ is also possible. Taking $(mn-1)(1)$ for $a_1a_2$ when $m=2$ also shows that $|\supp(S)|=m+1=3$ is possible when $m=2$.

\medskip

Note that, for groups of the form $G\cong C_m\oplus C_m$, all minimal zero-sum sequences of maximal length are of the form $S=e_1^{m-1}\prod_{i=1}^m(x_ie_1+e_2)$, where $\lbrace e_1,e_2\rbrace$ is a basis of $G$ with $\ord e_1=\ord e_2=m$ and $x_i\in\N_0$ with $\sum_{i=1}^mx_i\equiv 1\mod m$. In this situation, only Case 1 appears.

\end{document}